\documentclass[10pt,reqno]{article}
\usepackage[margin=1in]{geometry}
\usepackage{dsfont, amssymb,amsmath,amscd,latexsym, amsthm, amsxtra,amsfonts}
\usepackage{lineno}
\usepackage[all]{xy}
\usepackage[active]{srcltx}
\usepackage{tikz}
\usepackage{booktabs}
\usepackage[round,authoryear]{natbib}
\usepackage{bbm}
\usepackage{enumerate}
\usepackage{mathrsfs}
\usepackage{graphicx}
\usepackage{caption}
\usepackage{subcaption}
\usepackage{comment}
\usepackage{mathtools}
\usepackage{cases}
\usepackage{tcolorbox}

\tcbuselibrary{most}

\usetikzlibrary{calc,arrows}
\usepackage{verbatim}
\usepackage{color}
\usepackage{epstopdf}

\usepackage{bm}

\usepackage[title]{appendix}
\usepackage[pdfstartview=FitH, bookmarksnumbered=true,bookmarksopen=true, colorlinks=true, pdfborder={0 0 1}, citecolor=blue, linkcolor=blue,urlcolor=blue]{hyperref}

\hypersetup{colorlinks, citecolor=blue}
\hypersetup{
    colorlinks=true,
    linkcolor=blue,        
    anchorcolor=blue,      
    citecolor=blue,      
    filecolor=blue,        
    menucolor=blue,       
    runcolor=blue,         
    urlcolor=blue,         
}
\usepackage{graphics}
\graphicspath{{figures/}}
 
\newtheorem{theorem}{Theorem}[section]
\newtheorem{lemma}[theorem]{Lemma}
\newtheorem{proposition}[theorem]{Proposition}
\newtheorem{corollary}[theorem]{Corollary}
\newtheorem{definition}[theorem]{Definition}
\newtheorem{remark}[theorem]{Remark}
\newtheorem{assumption}[theorem]{Assumption}

\newcommand{\md}{\mathrm{d}}
\newcommand{\mE}{\mathbb{E}}
\newcommand{\mR}{\mathbb{R}}

\title{Equilibrium singular dividend control under ambiguity aggregation of heterogeneous discount rates}
\author{Yue Cao\thanks{Department of Mathematical Sciences, Tsinghua University, Beijing, China. \url{ caoyue24@mails.tsinghua.edu.cn}} 
\and Guohui Guan\thanks{Center for Applied Statistics and School of Statistics, Renmin University of China, Beijing 100872, China. \url{guangh@ruc.edu.cn}}
\and Zongxia Liang\thanks{Department of Mathematical Sciences, Tsinghua University, Beijing, China. \url{ liangzongxia@tsinghua.edu.cn}}
\and Xiaodong Luo\thanks{Department of Mathematical Sciences, Tsinghua University, Beijing, China. \url{ luoxd21@mails.tsinghua.edu.cn}}}
\date{}
\begin{document}

\maketitle
\vspace{-0.3in}
\begin{abstract}\noindent
 This paper studies a singular dividend control problem for a firm with heterogeneous shareholders whose discount rates follow a given distribution. The central planner aggregates expected discounted payoffs using an ambiguity aggregation function $\phi$, which captures shareholder heterogeneity and ambiguity attitudes but also leads to time inconsistency. To address this issue, we seek a time-homogeneous equilibrium dividend law characterized by a partition of the state space into waiting and dividend-paying regions. We provide a rigorous mathematical characterization by proving a verification theorem and deriving necessary conditions for the equilibrium law. We then analyze barrier-type equilibria, showing non-existence for a class of aggregation functions that includes power-type and logarithmic aggregation functions, and establishing existence and uniqueness under linear and exponential aggregation. In the linear case, the bounded-rate equilibrium is shown to converge to the singular barrier-type equilibrium as the dividend rate bound tends to infinity.  Numerical examples illustrate the effects of discount-rate heterogeneity and ambiguity aversion on the equilibrium barrier.

\vspace{0.1in}	
\noindent\textbf{Keywords}: 
singular dividend control; time inconsistency; equilibrium dividend law; heterogeneous discount rates; ambiguity aggregation

\vspace{0.1in}
\noindent\textbf{MSC 2020}: 93E20,  49L20, 91G50

\end{abstract}
\section{Introduction}
The optimal dividend problem, first proposed by \cite{de1957impostazione}, has become a cornerstone of actuarial and financial mathematics. It aims to determine a dividend strategy that maximizes the expected cumulative discounted dividends paid before ruin. This problem has been extensively studied under various surplus dynamics, including diffusion models (see, e.g., \cite{jeanblanc1995optimization,asmussen1997controlled}), the Cram\'er--Lundberg model (see, e.g., \cite{albrecher2009optimality}), the compound Poisson processes (see, e.g., \cite{azcue2012optimal}), the Lévy processes (see, e.g., \cite{noba2018optimal}) and surplus processes with model uncertainty (see, e.g., \cite{chakraborty2023optimal}). 
The problem is also analyzed under different operational mechanisms and constraints, including regime-switching (see, e.g., \cite{jiang2012optimal}), transaction costs (see, e.g., \cite{chen2021optimal}), bankruptcy level (see, e.g., \cite{ferrari2022optimal}), reinsurance (see, e.g., \cite{guan2023dynamic}) and drawdown constraint (see, e.g., \cite{guan2026optimal}).  In most existing studies, a key assumption is that the discount rate is constant over time and across shareholders; equivalently, the discount function is exponential with a fixed parameter.

However, shareholders often have heterogeneous time preferences in practice. Empirical evidence from experimental studies shows that individual discount rates vary significantly with age, education, income, and other demographic characteristics (see, e.g.,  \cite{warner2001personal,harrison2002estimating}).  \cite{frederick2002time} show that a single constant discount rate cannot capture the heterogeneity observed in actual behavior. Furthermore, \cite{weitzman2001gamma} demonstrates that even among expert economists, views on the appropriate discount rate are widely dispersed. Taken together, these findings suggest that shareholders of a firm are unlikely to share a common discount rate. Therefore, a central planner choosing a dividend policy on their behalf should account for the aggregated effect of heterogeneous discount rates.

\citet{ebert2020weighted} provide a recent approach to aggregating heterogeneous discount rates. They introduce the class of weighted discount functions, defined by $h(t)=\int_0^{\infty}e^{-rt}\md F(r)$, where $F$ is a distribution of discount rates. This representation directly captures group heterogeneity: different group members have different exponential discount rates, and the group’s aggregate discount function is given by their weighted average. 
\citet{deng2025time} adopt a different aggregation rule in the context of stopping problems. Instead of using a linear weighted average, they employ an ambiguity aggregation function \(\phi\), motivated by the smooth ambiguity framework of \citet{klibanoff2005smooth,klibanoff2009recursive}, to aggregate expected discounted payoffs under a distribution of discount rates. Their objective takes the form
 $$
\int_0^{\infty} \phi\left(\mE \left[e^{-r\tau }g(X_\tau^x)\right]\right)\md F_{\rho}(r),
 $$
where the concave aggregation function \(\phi\) captures the central planner’s attitude toward ambiguity arising from heterogeneous discount rates, and \(g(x)\) denotes the payoff function.

 In this paper, we extend this approach to the dividend problem. We model the firm’s shareholders as a heterogeneous group whose discount rates follow a distribution $F_{\rho}(r)$. The central planner aggregates expected discounted dividends through an ambiguity aggregation function $\phi$, leading to the objective
 $$
J(t,x;D)=
\int_0^\infty
\phi\left(
\mathbb E_{t,x}\Big[\int_t^{\tau^D}e^{-r(s-t)}\,dD_s\Big]
\right)\md F_\rho(r),
 $$
where \(D\) denotes the dividend process and \(\tau^D\) is the corresponding ruin time. The aggregation of heterogeneous discount rates, however, naturally gives rise to time inconsistency. Indeed, \cite{jackson2015collective} and \cite{millner2018time} show that any non-dictatorial collective preference satisfying unanimity is generically time-inconsistent. Moreover, \cite{ebert2020weighted} demonstrate that even a linear weighted average of exponential discount functions exhibits decreasing impatience, a classical source of time inconsistency. Thus, in our setting, the central planner’s aggregation of heterogeneous shareholder discount rates already generates time inconsistency, even under the linear aggregation function \(\phi(x)=x\).

A closely related dividend maximization problem with a random discount rate was studied by \cite{zhao2014dividend}, but under a regular control framework. In their model, the random discount rate makes the discount function non-exponential and renders the problem time-inconsistent, while dividends are restricted to a bounded rate, i.e., an absolutely continuous (regular) control. In contrast, our model allows the dividend process to be a general non-decreasing singular control and thus permits lump-sum payments. Moreover, we introduce an ambiguity aggregation function \(\phi\) to aggregate expected discounted dividends across heterogeneous discount rates, thereby allowing for a nonlinear transformation of the expected payoffs associated with different discount rates. Such nonlinear transformations of expectations are another well-known source of time inconsistency. A prominent example is the mean-variance criterion, where the objective depends on both the expectation and the variance of the payoff. This type of time inconsistency has been extensively studied in portfolio selection; see, e.g., \cite{Basak_dynamic_2010}, \cite{Bjork_2014}, and \cite{dai_dynamic_2021}. Recently, \cite{cao2026equilibrium} investigated a singular dividend control problem under the mean-variance criterion.

 In this paper, we combine two distinct sources of time inconsistency: the aggregation of heterogeneous shareholders’ discount rates, which arises even under linear aggregation, and the potentially nonlinear ambiguity aggregation function \(\phi\), which further distorts the objective. In the literature on time-inconsistent singular control, two equilibrium frameworks have emerged. The first defines equilibrium through singular control laws: the state space is partitioned into a waiting region and a dividend-paying region, and the candidate equilibrium control is generated by a Skorokhod reflection problem; see,  e.g., \cite{liang2024equilibria} and \cite{liang2025stackelberg}. The second defines equilibrium directly through the increments of singular controls, following a more conventional perturbation approach; see, e.g., \cite{dai2024dynamic} and \cite{cao2026equilibrium}. We adopt the first framework because our focus is on time-homogeneous equilibrium strategies, for which it is difficult to rigorously define a Markovian dividend strategy in terms of instantaneous singular dividend increments.

The main contributions of this paper are as follows. First, we introduce a framework for singular dividend control under ambiguity aggregation to accommodate heterogeneous shareholders with different discount rates. We provide a rigorous mathematical characterization by proving a verification theorem (Theorem \ref{thm:verification}) and deriving necessary conditions (Theorem \ref{thm:nes}) for the equilibrium singular dividend law and its value function. Unlike standard singular control problems, where a single variational inequality typically suffices, our equilibrium is governed by an aggregate weighted marginal condition involving both the ambiguity aggregation function $\phi$ and the distribution of heterogeneous discount rates. In particular, the verification theorem requires the additional variational inequalities \eqref{additional assump1} and \eqref{additional assump2}. These conditions arise because the potential nonlinearity of $\phi$ makes the central planner’s marginal weights over discount rates state-dependent: an immediate lump-sum dividend $d$ may change the weight $\phi'$ attached to each discount rate $r$. Consequently, the planner’s attitude toward different discount rates shifts after a dividend payment. This complication is absent from the mean-variance singular dividend problem of \cite{cao2026equilibrium}, where the variance is unaffected by an initial lump-sum payment.

Second, we study barrier-type equilibria and characterize the equilibrium barrier through an aggregate smooth-pasting condition in Theorem \ref{thm: barrrier point}. Using the necessary conditions in Theorem \ref{thm:nes}, we identify a class of ambiguity aggregation functions, including power-type and logarithmic forms, for which no barrier-type equilibrium can exist. A related non-existence result is obtained by \cite{bodnariu2025time}; however, the mechanism is different. In their model, non-existence is driven by the state-dependent inventory cost $\frac{1}{2}x^2$ in the objective, whereas in our framework it arises from state-dependent shifts in the central planner’s marginal preferences over shareholders with heterogeneous discount rates. In contrast, we establish the existence and uniqueness of barrier-type equilibria under linear aggregation $\phi(x)=x$ and exponential aggregation. In the linear case, we further show that the regular equilibrium dividend strategy in the bounded dividend-rate setting, under a discrete mixture of exponential discount rates, converges to the singular barrier-type equilibrium dividend law as the dividend rate bound $M$ tends to infinity. For the exponential case, we establish the existence of a unique barrier-type equilibrium dividend law under the mild condition on the central planner’s ambiguity aversion $\alpha$ stated in Assumption \ref{assum:alpha_constraint}. Finally, numerical examples illustrate how discount-rate heterogeneity and ambiguity aversion affect the equilibrium barrier. The numerical results indicate that the equilibrium barrier increases with the proportion of patient shareholders but decreases as the planner becomes more ambiguity-averse. By introducing a virtual discount rate, we identify two competing forces: discount-rate heterogeneity, which encourages patience, and ambiguity aversion, which drives impatience.

The remainder of this paper is organized as follows. Section \ref{sec:Model formulation} formulates the singular dividend control problem under heterogeneous discount rates and defines the equilibrium law. Section \ref{sec:Verification theorem} establishes the verification theorem, while Section \ref{sec:Necessary Condition} derives necessary conditions for equilibrium. Section \ref{sec:Barrier strategy in two examples} analyzes barrier-type equilibria, covering non-existence results, the linear aggregation case and its comparison with \cite{zhao2014dividend}, and the exponential aggregation case. Section \ref{sec:Numerical examples} presents numerical examples.

\section{Model Formulation}\label{sec:Model formulation}

Let $(\Omega,\mathcal F,\mathbb F,\mathbb P) =(\Omega,\mathcal F,\{\mathcal F_t\}_{t\ge0},\mathbb P)$
be a filtered probability space satisfying the usual conditions and supporting a one-dimensional standard Brownian motion $W=\{W_t\}_{t\ge0}$. We consider a dividend-paying firm. Let $D_t$ denote the cumulative dividends paid up to time $t$, and refer to $D=\{D_t\}_{t\ge0}$ as a dividend strategy. Under $D$, the firm’s surplus process evolves as
\begin{equation}\label{eq:surplus}
\md X_t^D=\mu(X_t^D)\,\md t+\sigma(X_t^D)\,\md W_t-\md D_t,\qquad X_{0-}^D=x>0,
\end{equation}
where $\mu(x)>0$ and $\sigma(x)>0$ are deterministic functions on $[0,+\infty )$. The ruin time is defined\footnote{Similar to \cite{cao2026equilibrium}, $\inf\{t\ge0:X_t^D\le0\}=\inf\{t\ge0:X_t^D<0\}$ almost surely as $\sigma>0$.} by
\begin{equation}\label{eq:ruin}
\tau^D:=\inf\{t\ge0:X_t^D\le0\}.
\end{equation}
We impose the following standard assumptions on $\mu$ and $\sigma$.
\begin{assumption}\label{assump: parameter}
    The drift function $\mu(x)$ and the volatility function $\sigma(x)$ satisfy the  linear growth and Lipschitz conditions: there exists a constant $C>0$ such that
    \begin{enumerate}
        \item for any $x\in [0,\infty)$, $|\mu(x)|+|\sigma(x)|\le C(1+|x|)$;
        \item for any $x_1,x_2\in [0,\infty)$, $|\mu(x_1)-\mu(x_2)|+|\sigma(x_1)-\sigma(x_2)|\le C|x_1-x_2|$.
    \end{enumerate}
\end{assumption}

We consider an optimal dividend problem on behalf of shareholders with heterogeneous discount rates. To capture this heterogeneity, we follow \cite{deng2025time} and assume that the discount rate $\rho$ is a random variable with distribution function $F_\rho$ on $[0,\infty)$, independent of the surplus process $X$. For a given dividend strategy $D$ and discount rate $r$, we define the total discounted dividends paid from time $t$ until the ruin time $\tau^D$ by
\[
Y^D(t;r):=\int_t^{\tau^D}e^{-r(s-t)}\md D_s,
\]
\noindent
and  set $Y^D(t;r)=0$ for all $t\ge \tau^D$. Shareholders with different discount rates may disagree on the optimal dividend strategy, as each discount rate can lead to a distinct recommendation. A central planner acting on behalf of heterogeneous shareholders must therefore aggregate these divergent preferences into a single objective. Following the collective decision-making approach of \cite{deng2025time}, we introduce an ambiguity aggregation function $\phi$ to aggregate expected discounted dividend payoffs across discount rates. The function $\phi$ captures the central planner’s attitude toward the ambiguity generated by heterogeneous discount rates: a concave $\phi$ reflects ambiguity aversion, as it gives relatively more weight to less favorable discounted payoff evaluations.
 Accordingly, the resulting objective function is
\begin{equation}\label{eq:objective}
J(t,x;D):=
\int_0^\infty
\phi\left(
\mathbb E_{t,x}\Big[\int_t^{\tau^D}e^{-r(s-t)}\,dD_s\Big]
\right)\md F_\rho(r)=\int_0^\infty
\phi\left(
\mE_{t,x} \left[Y^D(t;r)\right]
\right)\md F_\rho(r),
\end{equation}
where $\mE_{t,x}$ denotes expectation conditional on $X_{t-}=x>0$. We now impose the following standard assumptions on the ambiguity aggregation function $\phi$ and the distribution $F_{\rho}$ of discount rates.
\begin{assumption}\label{phiassump}
    The function $\phi\in C^2(\mR_+)$ is increasing and concave. Moreover, we assume that the support set $\operatorname{Supp}(F_{\rho})\subset [\underline{r},\bar r],\quad 0<\underline{r}\le\bar r<\infty $.
\end{assumption}
\subsection{Admissible dividend strategy and admissible dividend law}
\begin{definition}\label{def:clasadmis}
Given an initial time $t\ge0$ and an initial surplus $X_{t-}=x>0$,
 a dividend strategy $D=\{D_s\}_{s\ge t}$ is called admissible if:
\begin{enumerate}
\item  $D$ is  $\mathbb F$-adapted,  right-continuous, nondecreasing  with $D_{t-}=0$;
\item The controlled surplus process $X^D$ in \eqref{eq:surplus} is well defined on $[t,\tau^D)$;
\item For every $s\ge t$, the company cannot pay more dividends than it owns, 
$
\Delta D_s:=D_s-D_{s^-}\le X_{s-}^D;
$
\item No dividends are paid after ruin, i.e.,
$
D_s=D_{\tau^D}
$ for $ s\ge\tau^D$;
\item \label{item inte1} For $F_\rho$-a.e. $r>0$,
$
\mathbb{E}_{t,x}\left[Y^D(t;r)\right]<\infty,\quad a.s.;
$
\item \label{item inte2} The aggregated payoff is finite, i.e.,
\[
J(t,x;D):=\int_0^\infty \phi \bigl(\mathbb{E}_{t,x}\left[Y^D(t;r)\right]\bigr)\md F_\rho(r)<\infty,\quad a.s..
\]
\end{enumerate}
We denote by $\mathcal A(t,x)$ the set of all admissible dividend strategies starting from $t$ with initial surplus $X_{t-}=x$. 
\end{definition}

In our problem, because ambiguity aggregation $\phi$  over heterogeneous discount rates   induces time-inconsistency, we  do not seek a globally optimal strategy in the classical sense, but instead adopt the intrapersonal game approach of \cite{bjork2010general} to define a time-consistent equilibrium dividend strategy. As the problem is of singular control type, the standard perturbation approach for regular controls is not directly applicable. Following \cite{liang2024equilibria}, we adopt the notion of an equilibrium singular control law. Motivated by the connection between optimal singular control and Skorokhod reflection, we describe a dividend law through a partition of the state space into a no-transaction region and a pay region. Because all model parameters are time-homogeneous, we restrict attention to time-homogeneous equilibrium strategies; that is, the dividend decision depends only on the current surplus level $x$, rather than on the calendar time $t$. Consequently, the no-transaction region $\mathcal{NT}$ and the pay region $\mathcal{P}$ are defined as subsets of $\mR_+$ rather than $\mR_+\times[0,\infty)$.
 
We now introduce the admissible singular control law, adapted from \cite{liang2024equilibria} to the present dividend problem.

\begin{definition}[Admissible Singular Dividend Law] \label{def:admissible}
Suppose that  $\mathcal{D} := (\mathcal{NT}, \mathcal{P})$ is a partition of the state space $\mathbb{R}_+$, where $\mathcal{NT}$ is an open set (the no-transaction region) and $\mathcal{P}$ is a closed set (the pay region), such that $\mathcal{NT} \cup \mathcal{P} = \mathbb{R}_+$. A singular control law $\mathcal{D}$ is called \emph{admissible} if it satisfies the following conditions:

\begin{enumerate}
    \item  For any starting time $t \ge 0$ and initial surplus $X_{t-} = x > 0$, the following Skorokhod reflection problem:
    \begin{equation} \label{eq:skorokhod}
    \begin{cases}
    dX_s^D = \mu(X_s^D) \, \md s + \sigma(X_s^D) \, \md W_s - \md D_s, & s \in[t, \tau^D), \\
    X_s^D \in \overline{\mathcal{NT}}, & s \in[t, \tau^D), \\
    D_s =  \int_t^{s\wedge \tau^D} \mathbf{1}_{\{X_u^D \in \mathcal{P}\}} \, \md D_u, & s \in [t,\infty),\\
    D_{t-} = 0
    \end{cases}
    \end{equation}
    admits a unique strong, $\mathbb{F}$-adapted, right-continuous solution pair $(X^D, D)=(X^{t,x,\mathcal{D}},D^{t,x,\mathcal{D}})$. Here, $\tau^D := \inf\{s \ge t : X_s^D \le 0\}$ is the ruin time under the dividend strategy $D$. We call $D^{t,x,\mathcal{D}}$ the dividend strategy generated by the law $\mathcal{D}$.
    
    \item  The dividend strategy $D^{t,x,\mathcal{D}}$ generated by $\mathcal{D}$ satisfies Conditions \ref{item inte1} and \ref{item inte2} in Definition \ref{def:clasadmis}.
\end{enumerate}
\end{definition}
\begin{remark}
    Here $\{D_s^t\}_{s\ge t}$ denotes the cumulative dividend starting  $t-$. In fact, we have
    $$
D_s^{0,x,\mathcal{D}}-D_{t-}^{0,x,\mathcal{D}}= D_{s}^{t,X_{t-}^{x,\mathcal{D}},\mathcal{D}}.
    $$
    Moreover, by construction, dividends are paid only when the surplus process lies in the pay region. Hence, for any admissible singular dividend law $\mathcal{D}$ and any initial state $(t,x)$, the generated dividend strategy $D^{t,x,\mathcal{D}}$ is admissible in the sense of Definition \ref{def:clasadmis}.
\end{remark}

\subsection{Equilibrium Dividend Law}
Before defining an equilibrium dividend law, we introduce the perturbed strategy associated with a candidate singular dividend law.  Let $\mathcal{D}$ be a candidate singular dividend law. Fix $(t,x)$ with $x>0$, and assume that ruin has not occurred by time $t$. For $\varepsilon>0$, let $\delta\in\mathcal{A}(t,x)$ be a nondecreasing, right-continuous process on $[t,\infty)$. We assume that there exist $\tilde{\varepsilon}>0$ sufficiently small and a constant $M>0$, independent of $\varepsilon$, such that, for all $\varepsilon\le \tilde{\varepsilon}$,
\[
\delta(t-)=0,
\quad \delta(t)\in [0,x],\quad 
\delta\left(t+\varepsilon\right)-\delta(t)\le M \varepsilon
,
\]
where $M$ is an arbitrary constant independent of $\varepsilon$.
Given the initial dividend value $D_{t-}$, define the perturbed strategy $D^\varepsilon$ induced by the candidate equilibrium strategy $\mathcal{D}$ as follows:
\begin{equation}\label{eq:perturbed}
D_s^\varepsilon
=
\begin{cases}
 D_{t-}+\displaystyle\int_{t}^{s\wedge\tau^\varepsilon} \md\delta(u),
& t\le s<(t+\varepsilon)\wedge\tau^\varepsilon,\\[1.2em]
D^{\varepsilon}_{(t+\varepsilon)-}+D_s^{t+\varepsilon,X_{\left((t+\varepsilon)-\right)\wedge\tau^{\varepsilon}},\mathcal{D}},
& s\ge (t+\varepsilon)\wedge\tau^\varepsilon,
\end{cases}
\end{equation}
where $\tau^\varepsilon:=\tau^{D^\varepsilon}$ is the ruin time under the perturbed strategy $D^\varepsilon$. 
\begin{remark}
    The definition of this perturbed strategy is inspired by \cite{cao2026equilibrium} and \cite{liang2024equilibria}. Compared with the definition in \cite{cao2026equilibrium}, we  allow the perturbation process to have jumps, while still requiring the jump size after initial time to be $O(\varepsilon)$, which is consistent with \cite{liang2024equilibria}. Therefore, the jump at the initial time is directly incorporated into $\int_{t}^{s\wedge\tau^\varepsilon} \md\delta(u)$ in our definition. Furthermore, after a perturbation of duration $\varepsilon$ (assuming that ruin has not occurred), the dividend strategy is generated by the candidate equilibrium law, which can also be regarded as a closed‑loop control.
\end{remark}
We now define the equilibrium dividend law as follows.
\begin{definition}\label{def:equilibrium}
Let $\hat{\mathcal{D}}$ be an admissible singular dividend law. We call $\hat{\mathcal{D}}$ an equilibrium dividend law if, for every $(t,x)$ with $x>0$ and ruin not yet occurred by time $t$, for every perturbation  $\delta^\varepsilon$ that satisfies the above conditions and the corresponding perturbed strategy $D^{\varepsilon}$,
\begin{equation}\label{eq:equilibrium}
\liminf_{\varepsilon\downarrow0}
\frac{
J(t,x;\hat{\mathcal{D}})-J(t,x; D^\varepsilon)
}{\varepsilon}\ge0.
\end{equation}
The corresponding equilibrium value function is
$
V(t,x)=J(t,x;\hat {\mathcal{D}}).
$
\end{definition}

\section{Verification Theorem}\label{sec:Verification theorem}

This section establishes a verification theorem for the equilibrium singular dividend law and its value function in the time-inconsistent dividend problem \eqref{eq:objective}. Let $\mathcal{L}$ denote the infinitesimal generator of the uncontrolled surplus process, defined by
\begin{equation*}
\mathcal{L} f(x) := \mu(x) f_x(x) + \frac{1}{2}\sigma^2(x) f_{xx}(x), \quad   \forall \  f \in C^2(\mathbb{R}_+).
\end{equation*}

\begin{theorem}[Verification Theorem] \label{thm:verification}
Given a  function $y(x,r): [0,\infty) \times \operatorname{Supp}(F_{\rho})  \to [0,\infty) $. 
Suppose that $y(x,r)$ satisfies the following conditions:
\begin{enumerate}
    \item  For any $r \in \operatorname{Supp}(F_{\rho})$, the function $x \rightarrow y(x, r)$ is continuously differentiable in $\mR_+$ and twice continuously differentiable except possibly at the boundary $\partial \mathcal{NT}$, where its left and right second derivatives are well defined. 
    \item  Define  the
pay region $\mathcal{P}$ and no-transaction region $\mathcal{NT}$ respectively by
\begin{equation*}
    \mathcal{P} = \left\{x>0\,\,\bigg\vert\int_0^\infty \phi'\big(y(x, r)\big) \Big[ 1 - y_x(x, r) \Big] \md F_\rho(r) = 0\right\}, \quad \mathcal{NT}= \mR^+ \backslash\mathcal{P}.
\end{equation*}
\noindent
For any $r \in \operatorname{Supp}(F_{\rho})$, the function $x\rightarrow y(x, r)$ satisfies the following system
    \begin{equation} \label{eq:y_system}
    \begin{cases}
    \mathcal{L} y(x, r) - r y(x, r) = 0, & \text{for } x \in \mathcal{NT}, \\
    y_x(x, r) = 1, & \text{for } x \in \mathcal{P}
    \end{cases}
    \end{equation}
    with boundary condition $y(0,r)=0$ and suppose that the following variational inequalities hold for any $x>0$
    \begin{eqnarray}
       && \max\left\{ \int_0^\infty \phi'(y(x,r)) \big[ 1 - y_x(x,r) \big] \md F_\rho(r), \int_0^\infty \phi'(y(x,r)) \left(-ry(x,r)+\mathcal{L}y(x,r)\right) \md F_\rho(r)\right\}=0, \label{variational inequality}\\
    && \int_0^\infty \phi'(d+y(x,r)) \big[ 1 - y_x(x,r) \big] \md F_\rho(r)\le 0 ,\quad \text{for any } d>0, \label{additional assump1} \\
    &&\int_0^\infty \phi'(d+y(x,r)) \left(-ry(x,r)+\mathcal{L}y(x,r)\right) \md F_\rho(r)\le 0,\quad \text{for any } d>0.
     \label{additional assump2}
    \end{eqnarray}
    \item $\hat{\mathcal{D}}:= (\mathcal{NT},\mathcal{P})$ is an admissible singular dividend law.
    \item\label{assump: verimartingale} For any $r\in\operatorname{Supp}(F_{\rho})$, any initial time $t$, initial surplus $x>0$ and  any $k>0$, the following stochastic integrals have zero (conditional) expectation, 
    \[\int_t^{\tau_k} e^{-r(s-t)}\sigma(X_s^{t,x,\hat{D}}) y_x(X_s^{t,x,\hat{D}},r)\,dW_s.\]
    Here we define the bounded stopping time $\tau_k:=\tau^{\hat{D}}\wedge k$ for  $k>0$, where $\hat{D}=\hat{D}^{t,x,\hat{\mathcal{D}}}$ is the corresponding dividend strategy generated by $\hat{\mathcal{D}}$.
    \item \label{assump: trans} For any $r\in\operatorname{Supp}(F_{\rho})$, initial time $t$ and $k>0$,
    \[
    \lim_{k\to\infty}
\mathbb E_{t,x}\!\left[e^{-r(k-t)}y(X_{\tau_k}^{t,x,\hat{D}},r)\right]=0.
    \]
    \item $y(x, r)$ is non-decreasing in $x$ and non-increasing in $r$ and $y(x,r)>0$ for any $x>0$ and $r \in \operatorname{Supp}(F_{\rho})$.
    \item There exist constants $C, q > 0$ independent of $r$ such that for all $z \in \mathbb{R}_+$, the following growth bound holds
    \begin{equation}\label{condition: growth}
        |y(z, r)| + |y_x(z, r)| + |(\mathcal{L}y - ry)(z, r)| \leq C(1 + z^q).
    \end{equation}  
\end{enumerate}
Then $\mathbf{\hat{\mathcal{D}}}$ is an equilibrium singular dividend law and the corresponding equilibrium value function is given by\footnote{The boundary condition $y(0,r)=0$ is imposed only at the level of the auxiliary payoff, while $\phi(y(x,r))$ are evaluated for $x>0$.}
\begin{equation*}
    V(x) = \int_0^\infty \phi\big(y(x, r)\big) \md F_\rho(r),\quad \forall x>0.
\end{equation*}
\end{theorem}

\begin{remark}
The verification theorem highlights the time inconsistency induced by aggregating heterogeneous discount rates. Unlike classical singular control problems, where terms such as $1-V_x(x)$ or $\mathcal{L}V(x)-rV(x)$ govern the variational inequalities, the decision to pay a lump-sum dividend or to wait is determined here by the aggregate weighted marginal condition \eqref{variational inequality}. The endogenous weighting factor $\phi'\big(y(x,r)\big)$ depends on both the current state $x$ and the discount rate $r$, reflecting how the central planner’s ambiguity attitude reshapes the relative importance of different discount-rate evaluations.
\end{remark}
\begin{lemma}\label{lma:deviationestimation}
    Fix an initial state $(t,x)$ with $x>0$ and a perturbation process $\delta$. Let $X^\varepsilon$ be the controlled surplus process under the perturbed strategy $D^\varepsilon$ defined in \eqref{eq:perturbed}. Under Assumption \ref{assump: parameter}, for any $\varepsilon\le \tilde{\varepsilon}$ and $p>0$, there exists a constant $C$ independent of $\varepsilon$ such that
\begin{equation}\label{EX}
  \mE_{t,x}\left[ \sup_{s \in [t, t+\varepsilon]} |X^{\varepsilon}_s |^{2p} \right]\le C,\quad    \mE_{t,x}\left[ \sup_{s \in [t, t+\varepsilon]} |X^{\varepsilon}_s - \hat{x}|^4 \right] \le C \varepsilon^2,\quad a.s.,
\end{equation} 
 where $\hat{x}:=x-\delta_t$. Moreover, for any fixed $\kappa > 0$, we have
\begin{equation}\label{EX1}
        \mathbb{P}_{t,x}\left(\sup\limits_{s \in [t, t+\varepsilon]} |X^\varepsilon_s -\hat x| \geq \kappa\right) \le \frac{C \varepsilon^2}{\kappa^4},\quad a.s..
    \end{equation}
   
\end{lemma}
\begin{proof}
    The proof follows from standard SDE estimates. 
    Define the stopping time $\tau_n = \inf\{t \ge 0 : |X_t^{\varepsilon}| \ge n\}$, and in the following, $C$ denotes a generic positive constant that may change from line to line but remains independent of $n$ and $\varepsilon$. For the stopped process $X^{\varepsilon}_{s \wedge \tau_n}$, using the integral form of SDE \eqref{eq:skorokhod}, we have for a fixed $p>0$ and any $\varepsilon\le \tilde{\varepsilon}$,
\begin{equation*}
    |X^{\varepsilon}_{s \wedge \tau_n}|^{2p} \le C \left( \hat{x}^{2p} + \left| \int_t^{s \wedge \tau_n} \mu(X^{\varepsilon}_r) \md r \right|^{2p} + |\hat\delta_{s \wedge \tau_n}|^{2p} + \left| \int_t^{s \wedge \tau_n} \sigma(X_r^{\varepsilon}) \md W_r \right|^{2p} \right),\quad \forall t\le s\le t+\varepsilon.
\end{equation*}
Here $\hat{\delta}_s:=\delta_s-\delta_t$. Let $f_n(s) = E_{t,x}\left[ \sup\limits_{r \in [t, s \wedge \tau_n]} |X_r^{\varepsilon}|^{2p} \right]$. Using Hölder's inequality and Assumption \ref{assump: parameter}, we obtain
    \begin{equation*}
        \mE_{t,x}\left[ \sup_{r \in [t, s \wedge \tau_n]} \left| \int_t^r \mu(X_u^{\varepsilon}) \md u \right|^{2p} \right] \le \varepsilon^{2p-1} \int_t^{s} \mE_{t,x}\left[|\mu(X^{\varepsilon}_{r \wedge \tau_n})|^{2p}\right] \md r \le C \varepsilon^{2p-1} \int_t^{s} (1 + f_n(r)) \md r.
    \end{equation*}
    The given condition that $ \delta_r-\delta_t \le M(r-t)$ for $r\le s\le t+\tilde\varepsilon$ yields
    $
        \mE_{t,x}\left[ \sup_{r \in [t, s \wedge \tau_n]} |\hat \delta_r|^{2p} \right] \le M^{2p} \varepsilon^{2p} .
    $
    Using the Burkholder-Davis-Gundy inequality(BDG) and the Cauchy-Schwarz inequality, we have 
    \begin{equation*}
        \mE_{t,x}\left[ \sup_{s \in [t, s \wedge \tau_n]} \left| \int_t^r \sigma(X_u^{\varepsilon}) \md W_u \right|^{2p} \right] \le C \mE_{t,x}\left[ \left( \int_t^{s \wedge \tau_n} \sigma^2(X_r^{\varepsilon}) \md r \right)^p \right] \le C \varepsilon^{p-1} \int_t^{s} (1 + f_n(r)) \md r.
    \end{equation*} 
    Combining the above estimates, we obtain
$
    f_n(s) \le C(1 + \hat{x}^{2p}) + C \int_t^{s} f_n(r) \md s.
$
The Gronwall's inequality yields  for all $n$, $f_n(s) \le C e^{Cs}$. As the bound $C$ is independent of $n$, applying Fatou's Lemma, it follows that, as $n \to \infty$,
\begin{equation}
    \mE_{t,x}\left[ \sup_{s \in [t, t+\varepsilon]} |X_s|^{2p} \right] \le \liminf_{n \to \infty} f_n(t+\varepsilon) \le C < \infty, \quad a.s..
\end{equation}
We emphasize that the derived constant $C$ may depend on $\varepsilon$. However, we can consider $\varepsilon\le \tilde \varepsilon$ and fix a constant $C=C(\tilde \varepsilon)$.
\noindent Let $Y_s = X_s^{\varepsilon} - \hat{x} $. Then
\begin{equation*}
    \mE_{t,x}\left[ \sup_{s \in [t, t+\varepsilon]} |Y_s|^4 \right] \le C \left( \mE_{t,x}\left[ \sup_{s \in [t, t+\varepsilon]} \left| \int_t^{s} \mu(X_r^{\varepsilon}) \md r \right|^4 \right] + M^4\varepsilon^4 + \mE_{t,x}\left[ \sup_{s \in [t, t+\varepsilon]} \left| \int_t^{s} \sigma(X_r^{\varepsilon}) \md W_r \right|^4 \right] \right).
\end{equation*}
The expectations of the drift term and the dividend term are both bounded by $C \varepsilon^4$. For the diffusion term, applying the BDG inequality again yields
  {\small   \begin{equation*}
        \mE_{t,x}\left[ \sup_{s \in [t, t+\varepsilon]} \left| \int_t^{s} \sigma(X_r^{\varepsilon}) \md W_r \right|^4 \right] \le C \varepsilon \int_t^{t+\varepsilon} \mE_{t,x}\left[|\sigma(X_s^{\varepsilon})|^4\right] \md s \le C \varepsilon \int_t^{t+\varepsilon} (1 + \mE_{t,x}\left[|X_s^{\varepsilon}|^4\right]) \md s \le C \varepsilon^2.
    \end{equation*} }
Then $
    \mE_{t,x}\left[ \sup_{s \in [t, t+\varepsilon]} |X^{\varepsilon}_s - \hat{x}|^4 \right] \le C \varepsilon^2,\quad a.s.$. Thus,
\begin{equation*}
     \mathbb{P}_{t,x}\left[ \sup\limits_{s \in [t, t+\varepsilon]} |X_s^{\varepsilon} - \hat{x}| > \kappa \right] \le \frac{\mE_{t,x}\left[ \sup\limits_{s \in [t, t+\varepsilon]} |X_s^{\varepsilon} - \hat{x}|^4 \right]}{\kappa^4} \le \frac{C \varepsilon^2}{\kappa^4}.
\end{equation*} 
\end{proof}
\begin{proof}[{\bf Proof of Theorem \ref{thm:verification}}]{\ {\bf The proof is organized into the following two steps:}}

\vskip 5pt

\item \textbf{Step one: Representation of $y(\cdot,r)$ under the candidate law $\hat{\mathcal{D}}$}
\vskip 5pt
Fix $(t,x)$ and $r>0$. For simplicity, write
$
X_s:=X_s^{t,x,\hat D},\quad \tau:=\tau^{\hat D},
$
where $\hat{D}=\hat{D}^{t,x,\hat{\mathcal{D}}}$ is the singular
control generated by the law $\hat{\mathcal{D}}$. 
 As $y(\cdot,r)\in C^1(\mathbb R_+)\cap C^2(NT)$ and the second derivative admits finite limits on $\partial NT$,  applying the generalized It\^{o} formula to the semimartingale
$\{
e^{-r(s-t)}y(X_s,r),\  s\in[t,\tau_k]\}$,
we obtain
\begin{align*}
e^{-r(\tau_k-t)}y(X_{\tau_k},r)
&=y(x,r)
+\int_t^{\tau_k} e^{-r(s-t)}\bigl(\mathcal Ly(X_s,r)-r\,y(X_s,r)\bigr)\,\md s \\
&\quad
-\int_t^{\tau_k} e^{-r(s-t)} y_x(X_s,r)\,d\hat{D}^c_s
+\int_t^{\tau_k} e^{-r(s-t)}\sigma(X_s) y_x(X_s,r)\,\md W_s\\
&\quad + \sum_{s\in[t,\tau_k]}e^{-r(s-t)}\left(y(X_{s-}-\Delta \hat{D}_s,r)-y(X_{s-},r)\right),
\end{align*}
where $\hat{D}^c$ is the continuous part of $\hat{D}$ .
We  simplify each of the right side of the last equality as follows:\\
First, by admissibility of the singular dividend law, the reflected surplus process stays in $\overline{\mathcal NT}$ on $[t,\tau)$. Hence, by \eqref{eq:y_system}, we have $\bigl(\mathcal Ly(X_s,r)-r\,y(X_s,r)\bigr)=0$. Consequently, 
\[
\int_t^{\tau_k} e^{-r(s-t)}\bigl(\mathcal Ly(X_s,r)-r\,y(X_s,r)\bigr)\,\md s=0.
\]
Second, as $\hat{D}$ increases only when $X_s\in \mathcal P$ and $y_x(\cdot,r)=1$ on $\mathcal P$, it follows that
\begin{eqnarray*}
&&\int_t^{\tau_k} e^{-r(s-t)}y_x(X_s,r)\,\md \hat{D}^c_s
=\int_t^{\tau_k} e^{-r(s-t)}\,\md \hat{D}^c_s,
\\
&& \sum_{s\in[t,\tau_k]}e^{-r(s-t)}\left(y(X_{s-}-\Delta \hat{D}_s,r)-y(X_{s-},r)\right)=-\sum_{s\in[t,\tau_k]}e^{-r(s-t)} \Delta\hat{D}_s.
\end{eqnarray*}
Rearranging terms gives
\[
y(x,r)
=
\int_t^{\tau_k} e^{-r(s-t)}\,\md\hat D_s
+e^{-r(\tau_k-t)}y(X_{\tau_k},r)
-\int_t^{\tau_k} e^{-r(s-t)}\sigma(X_s) y_x(X_s,r)\,\md W_s.
\]
Taking the conditional expectation under $\mathbb E_{t,x}[\cdot]$, based on Assumption \ref{assump: verimartingale}, the stochastic integral has zero expectation. Consequently,
\[
y(x,r)
=
\mathbb E_{t,x}\!\left[\int_t^{\tau_k} e^{-r(s-t)}\,d\hat D_s\right]
+
\mathbb E_{t,x}\!\left[e^{-r(\tau_k-t)}y(X_{\tau_k},r)\right].
\]
Finally,  as $\tau_k\uparrow \tau$ and $D$ is nondecreasing, by the monotone convergence theorem, letting $k\to\infty  $ yields 
\[
\lim_{k\to\infty}
\mathbb E_{t,x}\!\left[\int_t^{\tau_k} e^{-r(s-t)}\,\md\hat D_s\right]
=
\mathbb E_{t,x}\!\left[\int_t^{\tau} e^{-r(s-t)}\,\md\hat D_s\right].
\]
Moreover, by the fact that $y(0,r)=0$,
\[
\lim_{k\to\infty}
\mathbb E_{t,x}\!\left[e^{-r(\tau_k-t)}y(X_{\tau_k},r)\right]=\lim_{k\to\infty}
\mathbb E_{t,x}\!\left[e^{-r(k-t)}y(X_{\tau_k},r)\right]\cdot \mathbf{1}_{\tau>k}=0.
\]
Therefore,
\[
y(x,r)=\mathbb E_{t,x}\!\left[\int_t^{\tau} e^{-r(s-t)}\,\md \hat D_s\right].
\]
The cumulative dividend payoff under $\hat{\mathcal{D}}$ is time homogeneous. Henceforth, we omit $t$ in the expected cumulative dividend payoff $Y$ and the corresponding aggregate payoff $J$.
\vskip 5pt 
\item \textbf{Step two: Verification of the equilibrium property for $\hat{\mathcal{D}}$}
\vskip 5pt 
Fix an initial state $(t,x)$ with $x>0$, and assume that ruin has not occurred by time $t$. Let $D^\varepsilon$ denote the perturbed strategy defined in \eqref{eq:perturbed}, which is generated by a perturbation process $\delta$ on $[t,t+\varepsilon)$. For simplicity, write
\[
X^{\varepsilon}_s:=X_s^{ D^{\varepsilon}},\qquad \tau^{\varepsilon}:=\tau^{ D^{\varepsilon}}.
\]

If $\delta_t=x$,  ruin occurs at $t$ immediately and the expected discounted dividend payoff equals $\phi(x)$. Using \eqref{additional assump1} and \eqref{additional assump2} yields
$$
\frac{\partial}{\partial d}\int_0^\infty \phi\big(d+y(x-d, r)\big) \md F_\rho(r)=\int_0^\infty \phi'(d+y(x-d,r)) \big[ 1 - y_x(x-d,r) \big] \md F_\rho(r)\le 0.
$$
Hence $\phi(x)\le V(x)$ and the expected payoff under the perturbed strategy is less than the payoff under the candidate equilibrium dividend strategy $\hat{D}$.

If $\delta_t<x$, then we consider the jumps of the initial perturbation separately from the subsequent $O(\varepsilon)$ perturbation. To be clear, we compare the payoffs corresponding to the following three strategies:
\begin{enumerate}
    \item Directly adopting the candidate equilibrium strategy $\hat{D}=D^{t,x,\hat{\mathcal{D}}}$ generated by $\hat{\mathcal{D}}$  ;
    \item Adopting the initial perturbed strategy $D^{0}$, namely first distributing dividends $\delta_t$ and resetting the initial surplus to $x-\delta_t$, then adopting the candidate equilibrium strategy $D^{t,x-\delta_{t},\hat{\mathcal{D}}}$ generated by $\hat{\mathcal{D}}$;
    \item Adopting the perturbed strategy $D^{\varepsilon}$.
\end{enumerate}
Then the difference between the payoff functions corresponding to the perturbed strategy and the candidate equilibrium strategy can be rewritten as
\begin{equation}\label{diffper}
\begin{aligned}
    J(x; D^\varepsilon) - J(x; \hat{D})
    =& \underbrace{J(x; D^\varepsilon) - J(x; D^0)}_{\text{Part 1}}+\underbrace{J(x; D^0)-J(x;\hat{D})}_{\text{Part 2}}.
\end{aligned}
\end{equation}
The first part denotes the expected payoff difference between the perturbed strategy $D^{\varepsilon}$ and the initial perturbed strategy $D^0$. The second part denotes the expected payoff difference between the initial perturbed strategy $D^0$ and the candidate equilibrium strategy $\hat D$. Let $\hat{x}=x-\delta_t>0$.
By Step one, the expected payoff under $D^0$ and $\hat{D}$ is 
$$
J(x; D^0)=\int_0^\infty \phi\big(\delta_t+y(\hat{x}, r)\big) \md F_\rho(r),\quad J(x;\hat{D})=\int_0^\infty \phi\big(y(x, r)\big) \md F_\rho(r).
$$
Using \eqref{additional assump1},  \eqref{additional assump2} and  the same argument when $\delta_t=x$, we have
\begin{eqnarray}\label{dd0}
J(x; D^0)-J(x;\hat{D})=\int_0^\infty \phi\big(\delta_t+y(\hat{x}, r)\big) \md F_\rho(r)-\int_0^\infty \phi\big(y(x, r)\big) \md F_\rho(r)\le 0.
\end{eqnarray}
It remains to show that the first part in \eqref{diffper} is bounded above by $o(\varepsilon)$.

For a given discount rate $r \in \operatorname{Supp}(F_{\rho})$, the expected discounted dividend payoff under the perturbed strategy $D^\varepsilon$ is 
  \begin{equation*}
Y^{D^\varepsilon}(x; r) = \mathbb{E}_{t,x}\left[ \int_t^{(t+\varepsilon)\wedge\tau^{\varepsilon}} e^{-r(s-t)} \md \delta_s + e^{-r\varepsilon\wedge (\tau^{\varepsilon}-t)} y(X^\varepsilon_{(t+\varepsilon-)\wedge\tau^{\varepsilon}}, r) \right].
\end{equation*}
Define a new perturbation process $\hat \delta:=\delta_s-\delta_t$. Using the same arguments as in Step one and  applying the generalized It\^{o} formula to $\{e^{-r(s-t)} y(X^\varepsilon_s, r)\}$,  we have
\begin{equation*}
\begin{aligned}
e^{-r\varepsilon\wedge (\tau^{\varepsilon}-t)} y(X^\varepsilon_{(t+\varepsilon-)\wedge\tau^{\varepsilon}}, r) &= y(\hat{x},r) + \int_t^{(t+\varepsilon)\wedge\tau^{\varepsilon}} e^{-r(s-t)} \big( \mathcal{L}y(X^\varepsilon_s, r) - r y(X^\varepsilon_s, r) \big) \md s \\
&\quad + \int_t^{(t+\varepsilon)\wedge\tau^{\varepsilon}} e^{-r(s-t)} \sigma(X_s^{\varepsilon})y_x(X^\varepsilon_s, r) \md W_s - \int_t^{(t+\varepsilon)\wedge\tau^{\varepsilon}} e^{-r(s-t)} y_x(X^\varepsilon_{s}, r) \md\delta^{c}_s\\
&\quad +\sum_{s\in[t,(t+\varepsilon)\wedge\tau^{\varepsilon}]}e^{-r(s-t)}\left(y(X_{s-}-\Delta \hat{\delta}^{\varepsilon}_s,r)-y(X_{s-},r)\right)  .
\end{aligned}
\end{equation*}
Taking the conditional expectation $\mathbb{E}_{t,x}[\cdot]$ on both sides of the last equation, the martingale term associated with $\md W_s$ vanishes. Let us define the difference in the expected dividend payoff as $\Delta_{\delta_t}^{\varepsilon} Y(x; r):= Y^{D^\varepsilon}(x; r) - y(\hat{x}, r)-\delta_t$. Rearranging the terms, we obtain
\begin{equation}\label{y diff}
\begin{aligned}
\Delta_{\delta_t}^{\varepsilon} Y(x; r) = & \mathbb{E}_{t,x} \bigg[ \int_t^{(t+\varepsilon)\wedge\tau^{\varepsilon}} e^{-r(s-t)} \big(\mathcal{L}y - ry\big)(X^\varepsilon_s, r) \md s + \int_t^{(t+\varepsilon)\wedge\tau^{\varepsilon}} e^{-r(s-t)} \big(1 - y_x(X^\varepsilon_{s}, r)\big) \md \delta^{c}_s  \\
& +  \sum_{s\in[t,(t+\varepsilon)\wedge\tau^{\varepsilon}]}e^{-r(s-t)}\int_0^{\Delta\hat\delta_s}(1-y_x(X^\varepsilon_{s-}-u, r))\md u\bigg] \\
= &\mathbb{E}_{t,x} \bigg[ \int_t^{(t+\varepsilon)\wedge\tau^{\varepsilon}}  \big(\mathcal{L}y - ry\big)(X^\varepsilon_s, r) \md s + \int_t^{(t+\varepsilon)\wedge\tau^{\varepsilon}}  \big(1 - y_x(X^\varepsilon_{s}, r)\big) \md \delta^{c}_s  \\
& +  \sum_{s\in[t,(t+\varepsilon)\wedge\tau^{\varepsilon}]}\int_0^{\Delta\hat\delta_s}(1-y_x(X^\varepsilon_{s-}-u, r))\md u\bigg] + o(\varepsilon).
\end{aligned}
\end{equation}
Here $\delta^{c}$ denotes the continuous part of $\delta$, which is the same as $\hat \delta^{c}$. The replacement of \(e^{-r(s-t)}\) by \(1\) in \eqref{y diff} only contributes an
\(o(\varepsilon)\) error. Indeed, $
|e^{-r(s-t)}-1|\le C\varepsilon 
$ as \(r\) is bounded and
\(s\in[t,t+\varepsilon]\).
By the polynomial growth condition \eqref{condition: growth} and Lemma \ref{lma:deviationestimation},  the total error caused by
removing the discount factors is \(O(\varepsilon^2)=o(\varepsilon)\).

The overall payoff functional aggregates the expected discounted dividend payoffs over the distribution $F_\rho$:
\begin{equation*}
J(x; D^\varepsilon) = \int_0^\infty \phi \Big( y(\hat{x},r) +\delta_t+ \Delta_{\delta_t}^{\varepsilon} Y(x;r) \Big) \text{d}F_\rho(r).
\end{equation*}
 By the concavity and differentiability of $\phi$, we have
\begin{equation*}
J(x; D^\varepsilon) \le \int_0^\infty \Big[ \phi(\delta_t+y(\hat{x},r)) + \phi'(\delta_t+y(\hat{x},r)) \Delta_{\delta_t}^{\varepsilon} Y(x;r) \Big] \text{d}F_\rho(r) .
\end{equation*}
Then, the first part of the difference in the objective functional can be estimated by
\begin{equation*}
\begin{aligned}
J(x; D^\varepsilon) - J(x; D^0) \le & \mathbb{E}_{t,x} \bigg[ \int_t^{(t+\varepsilon)\wedge\tau^{\varepsilon}} \int_0^\infty \phi'(\delta_t+y(\hat{x},r)) \big(\mathcal{L}y - ry\big)(X^\varepsilon_s, r) \md F_\rho(r) \md s \\
&+ \int_t^{(t+\varepsilon)\wedge\tau^{\varepsilon}} \int_0^\infty \phi'(\delta_t+y(\hat{x},r))\big(1 - y_x(X^\varepsilon_{s}, r)\big)\md F_\rho(r) \md \delta^{c}_s \\
&+\sum_{s\in[t,(t+\varepsilon)\wedge\tau^{\varepsilon}]}\int_0^{\Delta\hat\delta_s}\int_0^\infty \phi'(\delta_t+y(\hat{x},r))(1-y_x(X^\varepsilon_{s-}-u, r))\md F_\rho(r)\md u \bigg]+o(\varepsilon).
\end{aligned}
\end{equation*}
To bound the three terms on the right hand side of the last inequality, we fix $0<\kappa < \frac{1}{2}\hat{x}$ and define the large deviation event $E_\varepsilon = \left\{ \sup\limits_{s \in [t, t+\varepsilon]} |X^\varepsilon_s -\hat x| \geq \kappa \right\}$. 
Now we decompose the difference estimation into two parts:

\textbf{Part 1: Estimation on $E_\varepsilon^c$.}
On the set $E_\varepsilon^c$, we have $|X^\varepsilon_s - \hat x| < \kappa$ for all $s \in [t, t+\varepsilon]$. To proceed, we replace the constant multiplier $\phi'(\delta_t+y(\hat{x},r))$ with the stochastic multiplier $\phi'(\delta_t + y(X^\varepsilon_s, r))$. 
We show that this replacement introduces only an  $o(\varepsilon)$ error and thus we  aim to estimate the following three residue terms:
\begin{eqnarray*}\!\!\! \!\!\!
&&R_1(\varepsilon) = \mathbb{E}_{t,x} \left[ \int_t^{(t+\varepsilon)\wedge\tau^{\varepsilon}} \int_0^\infty \left| \phi'(\delta_t + y(X^\varepsilon_s, r)) - \phi'(\delta_t + y(\hat x, r)) \right| \cdot |(\mathcal{L}y-ry)(X^\varepsilon_s, r)| 1_{E_{\varepsilon^c}}\md F_\rho(r) \md s \right].\\
&&
R_2(\varepsilon) = \mathbb{E}_{t,x} \left[ \int_t^{(t+\varepsilon)\wedge\tau^{\varepsilon}} \int_0^\infty \left| \phi'(\delta_t + y(X^\varepsilon_s, r)) - \phi'(\delta_t + y(\hat x, r)) \right| \cdot |\big(1 - y_x(X^\varepsilon_{s}, r)\big)|1_{E_{\varepsilon^c}}\md F_\rho(r) \md \delta^{c}_s \right].\\
&&
R_3(\varepsilon) = \mathbb{E}_{t,x} \left[ \sum_{s\in[t,(t+\varepsilon)\wedge\tau^{\varepsilon}]}\int_0^{\Delta\hat\delta_s}\int_0^\infty \left| \phi'(\delta_t \!+\! y(X^\varepsilon_{s-}\!-\!u, r)) \!\!-\!\! \phi'(\delta_t \!+ \!y(\hat x, r)) \right| \cdot |\big(1 \!-\! y_x(X^\varepsilon_{s-}\!-\!u, r)\big)|1_{E_{\varepsilon^c}}\md F_\rho(r) \md u \right].
\end{eqnarray*}
By Assumption \ref{phiassump}, $\phi''$ is bounded in $[\delta_t+y(\frac{1}{2}\hat{x},\bar r), \delta_t+y(\frac{3}{2}\hat{x}, \underline{r})]$, which implies that $\phi'$ is Lipschitz continuous on this interval. Furthermore, by the regularity of $y(\cdot, r)$, the function $y$ is locally Lipschitz in $x$. Thus, there exists a constant $C > 0$ such that
\begin{eqnarray*}
&&|\phi'(\delta_t + y(X^\varepsilon_s, r)) - \phi'(\delta_t + y(\hat x, r))| \leq C |X^\varepsilon_s - \hat x|\le \sup_{s\in[t,t+\varepsilon]}|X_s^{\varepsilon}-\hat x|.
\\
&&
\left| \phi'(\delta_t + y(X^\varepsilon_{s-}-u, r)) - \phi'(\delta_t + y(\hat x, r)) \right|\le C |X^\varepsilon_{s-}-u - \hat x|\le \sup_{s\in[t,t+\varepsilon]}|X_s^{\varepsilon}-\hat x|.
\end{eqnarray*}
Using the growth condition \eqref{condition: growth} and the fact that $\sup_{s \in [t, t+\varepsilon]} |X^\varepsilon_s - \hat{x}| < \kappa$, $|(\mathcal{L}y-ry)(X^\varepsilon_s, r)|$ is uniformly bounded. Using the Hölder's inequality and \eqref{EX} in Lemma \ref{lma:deviationestimation}, we have 
\[
R_1(\varepsilon) \leq C \int_t^{t+\varepsilon} \mE_{t,x}\left[\sup_{s\in[t,t+\varepsilon]}|X_s^{\varepsilon}-\hat x|1_{E_{\varepsilon^c}}\right] ds = o(\varepsilon).
\]
As $\delta_{t+\varepsilon}-\delta_t\le M\varepsilon$ and $|1-y_x|$ are also bounded for the same reason,  the corresponding contributions to $R_2, R_3$ are bounded by
\[
R_2(\varepsilon)+R_3(\varepsilon) \le CM\varepsilon \mE_{t,x}\left[\sup_{s\in[t,t+\varepsilon]}|X_s^{\varepsilon}-\hat x|1_{E_{\varepsilon^c}}\right]=o(\varepsilon).
\]

\textbf{Part 2: Estimation on $E_\varepsilon$.} \  On  $E_{\varepsilon}$, we directly show that the difference term is $o(\varepsilon)$. Using the polynomial growth condition \eqref{condition: growth} on $y$, Hölder's inequality and Inequality \eqref{EX1} in Lemma \ref{lma:deviationestimation}, we have

\begin{eqnarray}   &&\mathbb{E}_{t,x}\left[\int_t^{(t+\varepsilon)\wedge\tau^{\varepsilon}}\int_0^\infty \phi'(\delta_t + y(\hat x, r))|(\mathcal{L}y-ry)(X^\varepsilon_s, r)| 1_{E_{\varepsilon}}\md F_\rho(r) ds\right] \nonumber\\
   && \le  C\phi'(\delta_t + y(\hat x, \bar r))\varepsilon\mE_{t,x}\left[\big(1+\sup_{s\in[t,t+\varepsilon]}|X_s^{\varepsilon}|^q\big)\cdot 1_{E_{\varepsilon}}\right] \nonumber\\
    &&\le  C\phi'(\delta_t + y(\hat x, \bar r))\varepsilon \mE_{t,x}\left[\big(1+\sup_{s\in[t,t+\varepsilon]}|X_s^{\varepsilon}|^q\big)^2\right]^{\frac{1}{2}}\mathbb{P}_{t,x}(E_\varepsilon)^{\frac{1}{2}}=o(\varepsilon). \label{P1}
\end{eqnarray}
Similarly,
\begin{eqnarray}
    &&\mathbb{E}_{t,x} \Bigg [\int_0^\infty  \phi'(\delta_t + y(\hat x, r)) \Bigg(\int_t^{(t+\varepsilon)\wedge\tau^{\varepsilon}} |\big(1 - y_x(X^\varepsilon_{s}, r)\big)|1_{E_{\varepsilon}} \md \delta^{c}_s\nonumber\\
    && +\sum_{s\in[t,(t+\varepsilon)\wedge\tau^{\varepsilon}]}\int_0^{\Delta\hat\delta_s}|\big(1 - y_x(X^\varepsilon_{s-}-u, r)\big)|1_{E_{\varepsilon}} \md u\Bigg) \md F_\rho(r) \Bigg]\nonumber\\
    &&\le  C\phi'(\delta_t + y(\hat x, \bar r))(M+1)\varepsilon \mE_{t,x}\Bigg[\big(1+\sup_{s\in[t,t+\varepsilon]}|X_s^{\varepsilon}|^q\big)^2\Bigg]^{\frac{1}{2}}\mathbb{P}_{t,x}(E_\varepsilon)^{\frac{1}{2}}=o(\varepsilon).\label{P2}
\end{eqnarray}
Then, using  Estimations on $ E_\varepsilon^c$ of {\bf  Part 1} and Inequalities \eqref{P1}-\eqref{P2}, we have 
\begin{equation*}
\begin{aligned}
J(x; D^\varepsilon) - J(x; D^0) \le& \mathbb{E}_{t,x} \left[ \int_t^{(t+\varepsilon)\wedge\tau^{\varepsilon}} G(\delta_t,X^\varepsilon_s)\cdot 1_{E_{\varepsilon^c}} \md s \right] + \mathbb{E}_{t,x} \left[ \int_t^{(t+\varepsilon)\wedge\tau^{\varepsilon}} H(\delta_t,X^\varepsilon_{s}) \cdot 1_{E_{\varepsilon^c}}\md\delta^{\varepsilon,c}_s \right] \\
& +\mathbb{E}_{t,x} \left[\sum_{s\in[t,(t+\varepsilon)\wedge\tau^{\varepsilon}]}\int_0^{\Delta\hat\delta_s} H(\delta_t,X^\varepsilon_{s-}-u)\cdot 1_{E_{\varepsilon^c}} \md u \right] + o(\varepsilon),
\end{aligned}
\end{equation*}
\noindent
where  the aggregate drift function $G(d,z)$ and aggregate marginal impact of dividend $H(d,z)$  are 
\begin{align*}
G(d,z) &:= \int_0^\infty \phi'(d+y(z,r)) \big(\mathcal{L}y(z, r) - r y(z, r)\big) \text{d}F_\rho(r)\le 0 ,\\
H(d,z) &:= \int_0^\infty \phi'(d+y(z,r)) \big(1 - y_x(z, r)\big) \text{d}F_\rho(r)\le 0.
\end{align*}
Here, the non-positivity is due to Conditions \eqref{additional assump1}-\eqref{additional assump2}. It follows that
\begin{eqnarray}\label{dd1}
J(x; D^\varepsilon) - J(x; D^0) \le 0 + o(\varepsilon)= o(\varepsilon).    
\end{eqnarray}
Using \eqref{dd0} and \eqref{dd1} yields
\begin{equation*}
J(x; D^\varepsilon) - J(x; \hat{D}) = J(x; D^\varepsilon)-J(x; D^0)+J(x; D^0)- J(x; \hat{D})\le o(\varepsilon).
\end{equation*}
Taking the limit infimum as $\varepsilon \downarrow 0$, we have 
\begin{equation*}
\liminf_{\varepsilon \downarrow 0} \frac{J(x; \hat{D}) - J(x; D^\varepsilon)}{\varepsilon} \ge 0.
\end{equation*}
 Thus the candidate law $\hat{\mathcal{D}}$ is an equilibrium dividend law. 
\end{proof}
\begin{remark}
    The initial perturbed strategy $D^0$ introduced in the proof characterizes the central planner's flexibility at the very beginning of the intervention. Specifically, $D^0$ represents the strategy under which, immediately after an initial lump-sum dividend $\delta_t$ is paid, the firm still retains the option to distribute further dividends at the same instant $t$ if the resulting state $x - \delta_t$ is still in the pay region according to the equilibrium law. However, due to the non-decreasing constraint on the dividend process $D$, this flexibility is one-sided: while further dividends can be added, any dividends already distributed cannot be reclaimed. 
    
    Notably, while $D^0$ is not an admissible perturbed strategy, it is crucial for comparing the equilibrium strategy and the perturbed strategy $D^\varepsilon$ (see Remark \ref{rmk: explainadditionasump}). Moreover, it is reasonable that $D^0$ is dominated by the equilibrium strategy. Otherwise, the central planner could increase the initial payoff via instantaneous dividend payouts, which contradicts the equilibrium.
\end{remark}
\begin{remark}\label{rmk: explainadditionasump}
Here we discuss Conditions \eqref{additional assump1} and \eqref{additional assump2}.
If we directly compute the difference between $Y^{D^{\varepsilon}}(x,r)$ and $y(x,r)$, then $\Delta Y(x; r):= Y^{D^\varepsilon}(x; r) - y(x, r)$ is the same as in \eqref{y diff}, except that $\Delta \hat{\delta}$ is replaced by $\Delta\delta$ and the difference in the objective functional
can be estimated by
\begin{equation}\label{eq:J_diff}
J(x; D^\varepsilon) - J(x; \hat{D}) \le \int_0^\infty \phi'(y(x,r)) \Delta Y(x; r) \md F_\rho(r) .
\end{equation}
Substituting the integral representation of $\Delta Y(x;r)$  into \eqref{eq:J_diff}, we have
\begin{equation*}
J(x; D^\varepsilon) - J(x; \hat{D}) \le \mathbb{E}_{t,x} \left[ \int_t^{t+\varepsilon} \int_0^\infty \phi'(y(x,r)) \big(\mathcal{L}y - ry\big)(X^\varepsilon_s, r) \md F_\rho(r) \md s + \cdots \right].
\end{equation*}
To proceed, we must replace the starting point evaluation $\phi'(y(x,r))$ with the random state evaluation $\phi'(y(X^\varepsilon_s,r))$.  However, if $\delta_t>0$, $|X^\varepsilon_s - x|$ does not converge to $0$ as $\varepsilon\rightarrow 0$ and such a replacement cannot be justified with an  $o(\varepsilon)$ error. 

This motivates the introduction of the intermediate strategy $D^0$. The variational inequality \eqref{variational inequality} is insufficient for our verification theorem. Therefore, the additional conditions \eqref{additional assump1} and \eqref{additional assump2} are needed. Moreover, Condition \eqref{additional assump1} is needed to ensure that the candidate equilibrium strategy $\hat{D}$ dominates the initial perturbed strategy $D^0$. 

The reason is that, when $\phi$ is nonlinear, an initial lump-sum dividend changes the marginal aggregation weights and hence the structure of the payoff. Although the mean-variance dividend problem investigated in \cite{cao2026equilibrium} involves the square of the cumulative discounted dividends, the variance term is unaffected by an initial deterministic lump-sum dividend. 
\end{remark}

\section{Necessary Condition for Equilibrium}\label{sec:Necessary Condition}
In this section, we derive a necessary condition for the equilibrium dividend law, which will be used later to seek barrier equilibria in special cases and to prove that a certain class of aggregations admits no barrier equilibria.
\begin{theorem}[Necessity of the Equilibrium Law]\label{thm:nes}
Let $\hat{\mathcal{D}} = (\mathcal{NT}, \mathcal{P})$ be an equilibrium singular dividend law and let $V(x)$ be the corresponding equilibrium value function. For a fixed discount rate $r \in \operatorname{Supp}(F_\rho)$, define $y(x, r)$ as the expected discounted dividend payoff generated by $\hat{\mathcal{D}}$,
\begin{equation*}
y(x, r) := \mathbb{E}_x \left[ \int_0^{\tau^{\hat{D}}} e^{-rs} \md\hat{D}_s \right].
\end{equation*}
Assume that Assumption \ref{phiassump} holds and  for any $r\in \operatorname{Supp}(F_{\rho})$, $y(\cdot, r) \in C^1(\mathbb{R}_+) \cap C^2(\mathcal{NT})$, the second derivative $y_{xx}(\cdot, r)$ has well-defined limits at the boundary $\partial \mathcal{NT}$. Moreover, we also assume that for any starting state $x \in \mR_+$ and any perturbation process $\delta$, there exists an $
\varepsilon_0=\varepsilon_0(x,\delta)$\footnote{In the undisturbed setting with exclusive use of equilibrium strategies, we also postulate the existence of an $\varepsilon_0$ for which the same condition holds.} such that the following integrability conditions hold
\begin{eqnarray} 
&&\mathbb{E}_{x} \left[ \int_0^{\varepsilon_0} e^{-2r(s-t)} \sigma^2 (X_s^{\varepsilon})\left(y_x(X^{\varepsilon}_s, r)\right)^2 \md s \right] < \infty,\label{nec:martingale_cond} \\ 
&& \mathbb{E}_{x} \left[ \sup_{\varepsilon \in (0, \varepsilon_0)} \frac{1}{\varepsilon} \left| \int_0^{\varepsilon} e^{-rs} \left( \mathcal{L}y(X_s^{\varepsilon}, r) - r y(X_s^{\varepsilon}, r) \right) \md s \right| \right] < \infty. \label{nec:uicond}
\end{eqnarray}
Then $y(x, r)$  satisfy the following conditions for any $r\in \operatorname{Supp}(F_{\rho})$.
\begin{enumerate}
    \item $y(x, r)$ satisfies the following system
    \begin{equation}\label{eq:HJB nes}
    \begin{cases} 
    \mathcal{L}y(x, r) - ry(x, r) = 0, & x \in \mathcal{NT}, \\
    y_x(x, r) = 1, & x \in \mathcal{P}
    \end{cases}
    \end{equation}
    with the boundary condition $y(0, r) = 0$.
    \item For any $x >0$, the variational inequality holds
    \begin{equation}\label{variational inequality nes}
    \max \left\{ \int_0^\infty \phi'(y(x, r))(1 - y_x(x,r)) \md F_\rho(r), \int_0^\infty \phi'(y(x, r)) (\mathcal{L}y(x, r) - ry(x, r)) \md F_\rho(r) \right\} = 0.
    \end{equation}
    \item For any $x \in \mathcal{P}$ such that the interval $[x, x+d] \subset \mathcal{P}$ for some $d > 0$, the aggregate drift condition holds
    \begin{equation}\label{additional nes}
    \int_0^\infty \phi'(d + y(x, r)) (\mathcal{L}y(x, r) - ry(x, r)) \md F_\rho(r) \leq 0.
    \end{equation}
\end{enumerate}
\end{theorem}

\begin{proof}
Notice that $y(0, r) = \mathbb{E}_0[\int_0^0 e^{-rs} \md D_s] = 0$. In the pay region $\mathcal{P}$, by Definition \ref{def:admissible}, one must immediately pay some dividends to push the surplus $x$ into the no-transaction region $\mathcal{NT}$. Thus, $y_x(x,r)=1$ in $\mathcal{P}$. 

  For $x\in\mathcal{NT}$ at initial time $0$, during the interval $[0, \hat \tau_\varepsilon)$, the surplus process $X$ follows the uncontrolled SDE: $dX_s = \mu(X_s) ds + \sigma(X_s) dW_s$. Here we define  $\hat \tau_\varepsilon = \inf\{s \ge 0 : X_s \notin \mathcal{NT}\} \wedge \varepsilon$ to be a bounded stopping time. Applying the generalized It\^{o} formula to the process $\{M_s := e^{-rs}y(X_s, r)\}$ on $[0, \hat \tau_\varepsilon]$, we obtain
\begin{align*}
e^{-r\hat \tau_\varepsilon}y(X_{\hat \tau_\varepsilon}, r) = y(x, r) &+ \int_0^{\hat \tau_\varepsilon} e^{-rs} \left( \mathcal{L}y(X_s, r) - ry(X_s, r) \right) \md s + \int_0^{\hat \tau_\varepsilon} e^{-rs} \sigma(X_s) y_x(X_s, r) \md W_s.
\end{align*}
Taking the conditional expectation $\mathbb{E}_{x}[\cdot]$ and utilizing Condition \eqref{nec:martingale_cond}, the stochastic integral term vanishes.   Definition of $y(x, r)$  and noting that no dividends are paid before $\hat\tau_\varepsilon$ yield 
\begin{equation*}
\mathbb{E}_{x} \left[ \int_0^{\hat \tau_\varepsilon} e^{-rs} \left( \mathcal{L}y(X_s, r) - ry(X_s, r) \right) \md s \right] = 0.
\end{equation*}
By the regularity condition of $y$ in $\mathcal{NT}$ and the uniform integrability \eqref{nec:uicond}, dividing by $\varepsilon$ and letting $\varepsilon \downarrow 0$  yield $\mathcal{L}y(x, r) - ry(x, r) = 0$ for all $x \in \mathcal{NT}$. 

To derive \eqref{variational inequality nes}, in light of \eqref{eq:HJB nes}, we only need to further prove the following inequalities:
\begin{eqnarray*}
    &&\int_0^\infty \phi'(y(x, r))(1 - y_x(x, r)) \md F_\rho(r)\le 0,\\
&&    \int_0^\infty \phi'(y(x, r)) (\mathcal{L}y(x, r) - ry(x, r)) \md F_\rho(r)\le 0.
\end{eqnarray*}
Consider a perturbed strategy\footnote{Here we add an additional superscript $d$ to emphasize the dependence of this strategy on the initial lump-sum payment $d$.} $D^{\varepsilon,d}$ starting from $x$ under which the firm first makes a lump-sum payment of size $0<d<x $, pays no dividends during $[0,\varepsilon)$  and thereafter pays dividends in accordance with the equilibrium dividend law $\hat{\mathcal{D}}$. Let $\varepsilon\rightarrow0$, then  the payoff functional is
$$
\lim_{\varepsilon\rightarrow0+}J(x; D^{\varepsilon,d}) = \lim_{\varepsilon\rightarrow 0+}\int_0^\infty \phi(d + \mathbb{E}_{x-d}\left[e^{-r(\varepsilon\wedge\tau_{\varepsilon})}y(X_{\varepsilon\wedge\tau_{\varepsilon}}, r)\right]) \md F_\rho(r) = \int_0^\infty \phi(d + y(x-d, r)) \md F_\rho(r).
$$
As $\hat{\mathcal{D}}$ is an equilibrium dividend law,  $J(x; \hat{\mathcal{D}}) - J(x; D^{\varepsilon,d}) \geq -o(\varepsilon)$ for sufficiently small $\varepsilon$, which implies $\int_0^\infty \phi( y(x, r)) \md F_\rho(r)-\int_0^\infty \phi(d + y(x-d, r)) \md F_\rho(r)\ge 0$ , thus 
\begin{equation*}
\frac{\partial}{\partial d} \left. \int_0^\infty \phi(d + y(x-d, r)) \md F_\rho(r) \right|_{d=0} = \int_0^\infty \phi'(y(x, r)) (1-y_x(x, r) ) \md F_\rho(r) \leq 0.
\end{equation*}
Let $D^\varepsilon$ be a strategy that forces a ``wait" period of length $\varepsilon$ at $x$ before resuming $\hat{\mathcal{D}}$. Then the expected payoff under $D^\varepsilon$ for a fixed $r$ is $Y^\varepsilon(x; r) = \mathbb{E}_{x}[e^{-r(\varepsilon\wedge\tau_\varepsilon)} y(X^{\varepsilon}_{\varepsilon\wedge\tau_{\varepsilon}}, r)]$. Using Taylor expansion yields
\begin{equation*}
J(x; D^\varepsilon)-J(x;\hat{\mathcal{D}}) = \int_0^\infty \phi^{'}(y(x, r))\left(\mathbb{E}_{x}\left[e^{-r(\varepsilon\wedge\tau_\varepsilon)}y(X^{\varepsilon}_{\varepsilon\wedge\tau_\varepsilon},r)\right]-y(x,r)\right)  \md F_\rho(r)+Res.
\end{equation*}
The second-order residue term satisfies
$$
Res\le\int_0^\infty\frac{1}{2}\sup\limits_{z\in \mathcal{Z}} |\phi^{''}(z)|\left(\mathbb{E}_{x} \left[ \left| \int_0^{\varepsilon\wedge\tau_\varepsilon} e^{-rs} \left( \mathcal{L}y(X_s^{\varepsilon}, r) - r y(X_s^{\varepsilon}, r) \right) \md s \right| \right]\right)^2 \md F_\rho(r)=o(\varepsilon).
$$
Define  a closed interval: 
$$\mathcal{Z}:=\left[y(x,r)-\left|\mathbb{E}_{x}\left[e^{-r(\varepsilon\wedge\tau_\varepsilon)}y(X^{\varepsilon}_{\varepsilon\wedge\tau_\varepsilon},r)\right]-y(x,r)\right|,y(x,r)+\left|\mathbb{E}_{x}\left[e^{-r(\varepsilon\wedge\tau_\varepsilon)}y(X^{\varepsilon}_{\varepsilon\wedge\tau_\varepsilon},r)\right]-y(x,r)\right|\right].$$
As $\mathbb{E}_{x}\left[e^{-r(\varepsilon\wedge\tau_\varepsilon)}y(X^{\varepsilon}_{\varepsilon\wedge\tau_\varepsilon},r)\right]-y(x,r)=\mathbb{E}_{x} \left[  \int_0^{\varepsilon\wedge\tau_\varepsilon} e^{-rs} \left( \mathcal{L}y(X_s^{\varepsilon}, r) - r y(X_s^{\varepsilon}, r) \right) \md s  \right]$ is  an $O(\varepsilon)$ term, $\mathcal{Z}$ is a compact subset of $\mR_+$ for sufficiently small $\varepsilon$.
Applying the equilibrium condition $\lim_{\varepsilon \downarrow 0} \frac{J(x; \hat{\mathcal{D}}) - J(x; D^\varepsilon)}{\varepsilon} \geq 0$ yields
\begin{equation*}
\int_0^\infty \phi'(y(x, r)) (\mathcal{L}y(x, r) - ry(x, r)) \md F_\rho(r) \leq 0.
\end{equation*}

Finally, we consider the necessity of \eqref{additional nes} for $x \in \mathcal{P}$. Suppose $[x-d, x] \subseteq \mathcal{P}$. Due to $y_x = 1$ in $\mathcal{P}$, we have $y(x; r) = d + y(x-d; r)$. This identity implies that the equilibrium strategy $\hat{\mathcal{D}}$ at $x$ is identical in payoff to the initial perturbed strategy $D^0$ which pays $d$ immediately and then follows $\hat{\mathcal{D}}$ from $x-d$:
\begin{equation*}
J(x; \hat{\mathcal{D}}) = \int_0^\infty \phi(y(x, r)) \md F_\rho(r) = \int_0^\infty \phi(d + y(x-d, r)) \md F_\rho(r) = J(x; D^0).
\end{equation*}
Now, consider the perturbed strategy $D^{\varepsilon,d}$ we defined before. The equilibrium condition $J(x; \hat{\mathcal{D}}) \geq J(x; D^{\varepsilon,d})+o(\varepsilon)$ is equivalent to $J(x; D^0) \geq J(x; D^{\varepsilon,d})+o(\varepsilon)$, which yields
\begin{equation*}
\int_0^\infty \phi'(d+ y(x-d, r)) (\mathcal{L}y(x-d, r) - ry(x-d, r)) dF_\rho(r) \leq 0.
\end{equation*}
\end{proof}
\begin{remark}
    The martingale condition \eqref{nec:martingale_cond} and the uniform integrability condition \eqref{nec:uicond} are similar to conditions (4.1), (4.2) in \cite{liang2024equilibria}, which are classical assumptions in necessity analysis. These two conditions play a crucial role in the proof. The first ensures that the stochastic integral vanishes upon taking the conditional expectation, while the second justifies the interchange of limits as $\varepsilon\rightarrow0$ and also ensures the second-order term in the Taylor expansion is of order $o(\varepsilon)$.
\end{remark}

\section{Barrier-Type Equilibrium}\label{sec:Barrier strategy in two examples}
In this section, we investigate the existence and non-existence of barrier-type equilibrium strategies. A singular dividend law $\hat{\mathcal{D}} = (\mathcal{NT}, \mathcal{P})$ is called barrier-type if there exists a constant $k^* > 0$ such that the no-transaction region is $\mathcal{NT} = (0, k^*)$ and the pay region is $\mathcal{P} = [k^*, \infty)$. Under such a law, we first show that, for any $r \in \text{Supp}(F_\rho)$, the expected discounted dividend $y(x, r)$ satisfies the following ODE:
 \begin{equation} \label{eq:general_ode}
 \left\{
 \begin{aligned}
     & \mathcal{L}y(x,r) - ry(x, r) = 0, \quad x \in [0, k^*),\\
     & y_x(x,r)=1,\quad x\in [k^*,\infty)
 \end{aligned}
 \right.
 \end{equation}
 with $y(0; r) = 0$ and the smooth-pasting condition $\lim\limits_{x\rightarrow k^*}y_x(x; r) = 1$. 
 \begin{theorem}\label{thm:conditionverify}
     The expected discounted dividend $y(x, r)$ satisfies \eqref{eq:general_ode} under a barrier-type dividend law. Moreover, any barrier-type dividend law is admissible and $y$ satisfies the integrability conditions  \eqref{nec:martingale_cond} and \eqref{nec:uicond}.
 \end{theorem}
 \begin{proof}
    Assume that $\mathcal{NT}=(0,k^*)$ for a fixed $k^*>0$. Using Theorem 4.1 in \cite{tanaka1979stochastic}, there exists a unique $\mathbb{F}$-adapted solution of the Skorokhod equation \eqref{eq:skorokhod} as $\mu$ and $\sigma$ satisfy Assumption \ref{assump: parameter} and $\mathcal{NT}$ and $\mathcal{P}$ are convex under a barrier-type dividend law. Let $y(x, r)$ be the solution to the ODE \eqref{eq:general_ode}. By construction, $y(\cdot, r)$ is $C^2$ on the no-transaction region $(0, k^*)$ and satisfies the smooth-pasting condition $\lim\limits_{x \to k^*} y_x(x; r) = 1$, which ensures that $y(\cdot, r)$ is $C^1$ across the boundary $k^*$. As $y_x(x, r) = 1$ for all $x \geq k^*$, the first derivative $y_x$ is continuous and globally bounded on $\mathbb{R}_+$. Given the linear growth condition of $\sigma(x)$ in Assumption \ref{assump: parameter}, we have
     $$
     \mathbb{E}_{t,x}\left[\int_t^{\tau_k} e^{-2r(s-t)}\sigma^2(X_s^{t,x,\hat{D}})y_x^2(X_s^{t,x,\hat{D}}, r)\md s\right]\le C\mE_{t,x}\left[\int_t^{\tau_k} (1+|X_s^{t,x,\hat{D}}|^2)\md s\right]<\infty,
     $$
    which guarantees Condition \ref{assump: verimartingale} of Theorem \ref{thm:verification}. Under the barrier strategy $\hat{\mathcal{D}}$, the reflected surplus process $X^{\hat{\mathcal{D}}}$ is confined to the compact interval $[0, k^*]$, where $y(x, r)$ is continuous and thus uniformly bounded. This boundedness ensures that 
\[
\lim_{k \to \infty} \mathbb{E}_{t,x} \left[ e^{-r(k-t)} y(X_{\tau_k}^{\hat{\mathcal{D}}}, r) \right] = 0
\]
is satisfied, fulfilling Condition \ref{assump: trans} of Theorem \ref{thm:verification}. Using the same argument as in the proof of Theorem \ref{thm:verification}, we can show that $y$ is exactly the expected discounted dividends under $\hat{\mathcal{D}}$.  Conditions \ref{item inte1} and \ref{item inte2} in Definition \ref{def:clasadmis} can also be easily verified as we have proved the expected discounted dividends is exactly $y(x,r)$ and $y(x,r)=\mE_x\left[\int_0^{\tau^{\hat D}}e^{-rs} \md \hat{D}_s\right]$ is continuous in $r$ with $r\in[\underline{r},\overline{r}]$. Thus, any barrier-type dividend law is admissible.

 Also by the global boundedness of $y_x$ and the linear growth of $\sigma$, we can verify that 
 $$
\mathbb{E}_x\left[\int_0^{\varepsilon_0} e^{-2rs}\sigma^2(X^{\varepsilon}_s)y_x^2(X^{\varepsilon}_s, r)\md s\right]\le C\mE_x\left[\int_0^{\varepsilon} (1+|X_s^{\varepsilon}|^2)\md s\right]<\infty,
 $$
 for any perturbed process $X^\varepsilon$ and the martingale condition \eqref{nec:martingale_cond} holds. Finally, for Condition \eqref{nec:uicond}, we note that the term $\mathcal{L}y(x, r) - ry(x, r)$ vanishes for $x \in (0, k^*)$ and, for $x \geq k^*$, takes the form
\[
(\mathcal{L}y - ry)(x, r) = \mu(x) - r(x - k^* + y(k^*, r)).
\]
Because both $\mu(x)$ and $y(x, r)$ exhibit at most linear growth,  $|\mathcal{L}y - ry| \leq C(1+|x|)$. Consequently, the uniform integrability requirement is satisfied as
\[
\mathbb{E}_x \left[ \sup_{\varepsilon \in (0, \varepsilon_0)} \frac{1}{\varepsilon} \left| \int_0^{\varepsilon} e^{-rs} (\mathcal{L}y - ry)(X_s^\varepsilon, r) \md s \right| \right] \leq \mathbb{E}_x \left[ \sup_{s \in [0, \varepsilon_0]} C(1 + |X_s^\varepsilon|) \right] < \infty,
\]
where the last inequality follows from standard SDE estimates for the supremum of the surplus process over a small time interval.
 \end{proof}
 
By Theorem \ref{thm:conditionverify}, we can use the necessary conditions derived in Theorem \ref{thm:nes}, based on which we next provide a necessary characterization of the equilibrium barrier point $k^*$ for the surplus process  \eqref{eq:surplus}.

\begin{theorem}\label{thm: barrrier point}
    Suppose that $\hat{\mathcal{D}}$ is a barrier-type equilibrium law with barrier $k^* > 0$. Under the assumptions of Theorem \ref{thm:nes}, the barrier $k^*$ must satisfy the following aggregate smooth-pasting condition
 \begin{equation} \label{eq:aggregate_smooth_pasting}
G(k^*):=\int_0^\infty \phi'(y(k^*, r)) \left( \mu(k^*) - r y(k^*, r) \right) \md F_\rho(r) = 0.
\end{equation}
\end{theorem}
\begin{proof}
First, by Theorem \ref{thm:nes}, 
\begin{equation}\label{eq:barrier1}
\int_0^\infty \phi'(y(k^*, r)) \left( \mu(k^*) - r y(k^*, r) \right) \md F_\rho(r) \leq 0,
\end{equation}
where we have just  used \( y_x(k^*, r) = 1 \) and \( y_{xx} = 0 \) on \( \mathcal P \). Also by  Theorem \ref{thm:nes}, 
\begin{equation}\label{eq:barrierauxil}
    \int_0^\infty \phi'(y(x, r)) (1 - y_x(x, r)) \md F_\rho(r) \le 0,\quad\forall x\in(0,k^*).
\end{equation}
As
\[
\int_0^\infty \phi'(y(k^*; r)) (1 - y_x(k^*; r)) \md F_\rho(r) = 0
,\]
 we differentiate \eqref{eq:barrierauxil} with respect to \( x \) at $k^*$, which requires
\[
\left. \int_0^\infty \left[ \phi''(y) \cdot y_x (1 - y_x) + \phi'(y) \cdot (-y_{xx}) \right] \md F_\rho(r) \right|_{x=k^*-} \geq 0,
\]
that is,
\[
\left. \int_0^\infty \phi'(y) y_{xx} \md F_\rho(r) \right|_{x=k^*-} \leq 0.
\]
By \( \mathcal{L}y - r y = 0 \) in $\mathcal{NT}$, \( y_{xx}(k^*, r) = \frac{2}{\sigma^2(k^*)} (r y(k^*, r) - \mu(k^*)) \). Substituting this, we obtain
\begin{equation}\label{eq: barrier2}
\int_0^\infty \phi'(y(k^*, r)) \cdot (\mu(k^*) - r y(k^*, r)) \md F_\rho(r) \geq 0.
\end{equation}
Combining \eqref{eq:barrier1} and \eqref{eq: barrier2}, we obtain \eqref{eq:aggregate_smooth_pasting}.
\end{proof}
\begin{remark} \label{rem:aggregate_smooth_pasting}
We refer to \eqref{eq:aggregate_smooth_pasting} as the aggregate smooth-pasting condition for the following reason. The $C^2$ regularity of the equilibrium value function $V(x)$ at the barrier $k^*$ is equivalent to the matching of its second-order derivatives from both sides, i.e., $V''(k^*-) = V''(k^*+)$. Differentiating $V(x) = \int_0^\infty \phi(y(x,r)) \md F_\rho(r)$ twice, we have
\begin{equation*}
    V''(x) = \int_0^\infty \left[ \phi''(y(x,r)) y_x^2(x,r) + \phi'(y(x,r)) y_{xx}(x,r) \right] \md F_\rho(r).
\end{equation*}
As $y(\cdot, r) \in C^1(\mathbb{R}_+)$ and $y_x(k^*, r) = 1$, the first term in the integrand is continuous at $k^*$. Thus, $V''(k^*-) = V''(k^*+)$ if and only if the aggregate second-order term $\int_0^\infty \phi'(y(x,r)) y_{xx}(x,r) \md F_\rho(r)$ is continuous at $k^*$. In the pay region $\mathcal{P} = [k^*, \infty)$, we have $y(x,r) = x - k^* + y(k^*, r)$, which implies $y_{xx}(x,r) = 0$ for $x > k^*$. On the other hand, in the no-transaction region $\mathcal{NT}$, the relation $\mathcal{L}y - ry = 0$ implies that $y_{xx}(k^*-, r) = \frac{2}{\sigma^2(k^*)} (ry(k^*, r) - \mu(k^*))$. Therefore, the requirement that the left limit of the aggregate term vanishes, i.e.,
\begin{equation*}
    \lim_{x \uparrow k^*} \int_0^\infty \phi'(y(x,r)) y_{xx}(x,r) \md F_\rho(r) = 0
\end{equation*}
is precisely equivalent to the condition $\int_0^\infty \phi'(y(k^*, r)) (\mu(k^*) - ry(k^*, r)) \md F_\rho(r) = 0$ as stated in \eqref{eq:aggregate_smooth_pasting}.
\end{remark}
\subsection{Non-existence for a Class of Aggregation Functions}\label{sec:non-exist}
Throughout this subsection, we assume that $\operatorname{Supp}(F_{\rho})$ is not a singleton (i.e., $F_{\rho}$ is non-degenerate). Under this assumption, we prove that a barrier-type equilibrium dividend law does not exist for a class of ambiguity aggregation functions, including power and logarithmic functions. We first recall the definition of long-tailed functions.
\begin{definition}
    A positive measurable function $f:\mR_+\rightarrow\mR_+$ is said to be  long-tailed if for every fixed $y \in \mathbb{R}$, it satisfies:
\begin{equation}
    \lim_{x \to \infty} \frac{f(x + y)}{f(x)} = 1.
\end{equation}
\end{definition}
\noindent 
Define the class $\mathcal{S}$ by
$$
\mathcal{S}=\left\{\phi : \phi \,\text{is strictly concave and satisfies Assumption  \ref{phiassump}}, \phi^{\prime}\, \text{is long-tailed}\right\}.
$$
The power-type form $\phi(x)=x^{1-\gamma}$ with $\gamma\in(0,1)$ and the logarithmic form $\phi(x)=\log x$ both belong to $\mathcal{S}$. We first examine the monotonicity of $y$ and $ry$ with respect to $r$.
\begin{lemma}\label{lma:mono}
For a fixed barrier $k^* > 0$, the expected discounted dividend $y(k^*, r)$ is strictly decreasing in $r$, while  $ry(k^*, r)$ is strictly increasing in $r$.
\end{lemma}

\begin{proof}
Let $\psi(x, r)$ be the unique solution to the initial value problem $\mathcal{L}\psi - r\psi = 0$ with $\psi(0, r) = 0$ and $\psi_x(0, r) = 1$. By the linearity of the equation, any solution satisfying $y(0, r) = 0$ is of the form $y(x, r) = C(r)\psi(x, r)$. Applying the boundary condition $y_x(k^*, r) = 1$, we have
$
    y(x, r) = \frac{\psi(x, r)}{\psi_x(k^*, r)},
$
based on which we  respectively prove  monotonicity of $y(k^*, r)$ and monotonicity of $ry(k^*, r)$ as follows.
\vskip 5pt \noindent
\textbf{1. Monotonicity of $y(k^*, r)$.} \ 
Let $r_2 > r_1 > 0$. Define $y_1(x) = y(x, r_1)$, $y_2(x) = y(x, r_2)$, and $w(x) = y_1(x) - y_2(x)$. Then  the function $w$ satisfies $w(0) = 0$, $w_x(k^*) = 0$, and the differential inequality:
\begin{equation*}
    \frac{1}{2} \sigma^2 w_{xx} + \mu w_x - r_1 w = -(r_2 - r_1) y_2 < 0,
\end{equation*}
where $y_2 > 0$ on $(0, k^*]$. Suppose that $w(x)$ attains a negative minimum at $x_0 \in (0, k^*]$. 
\begin{itemize}
    \item If $x_0 \in (0, k^*)$, then $w_x(x_0) = 0$ and $w_{xx}(x_0) \ge 0$. Substituting into the equation gives $-r_1 w(x_0) < 0$, which implies $w(x_0) > 0$, a contradiction.
    \item If $x_0 = k^*$, then $w_x(k^*) = 0$ and $w_{xx}(k^*) \ge 0$ (as a limit from the left). Substituting into the equation at $x=k^*$ yields $\frac{1}{2}\sigma^2(k^*)w_{xx}(k^*) - r_1 w(k^*) < 0$. As $w_{xx} \ge 0$ and $w < 0$, the left side is strictly positive, leading to a contradiction.
\end{itemize}
Thus, $w(x) \ge 0$ on $[0, k^*]$, implying $y(k^*, r_1) \ge y(k^*, r_2)$. Moreover, it is easy to verify that $y(k^*, r_1) = y(k^*, r_2)$ is impossible.
\vskip 5pt\noindent
\textbf{2. Monotonicity of $r y(k^*, r)$.} \ 
Define $g(x, r) = \frac{r \psi(x, r)}{\psi_x(x, r)}$, then $g(k^*, r) = r y(k^*, r)$. Differentiating $g$ with respect to $x$ and substituting $\psi_{xx} = \frac{2}{\sigma^2}(r\psi - \mu\psi_x)$, we obtain 
\begin{equation*}
    g_x = r + \frac{2\mu}{\sigma^2} g - \frac{2}{\sigma^2} g^2.
\end{equation*}
Differentiating this equation with respect to $r$ and setting $v = \frac{\partial g}{\partial r}$ yield
\begin{equation*}
    v_x = 1 + \left( \frac{2\mu - 4g}{\sigma^2} \right) v, \quad v(0, r) = 0.
\end{equation*}
Thus the solution  $v(x, r)$ satisfies
\begin{equation*}
    v(x, r) = \int_0^x \exp \left( \int_\xi^x \frac{2\mu(z) - 4g(z, r)}{\sigma^2(z)} dz \right) d\xi> 0,\quad  \forall \ x>0.
\end{equation*}
\end{proof}

\begin{proposition}\label{prop:neutral}
Let $k^*$ be the equilibrium barrier point for the dividend problem with an ambiguity aggregation function $\phi\in\mathcal{S}$. Let $G_{RN}(k):= \int_0^\infty (\mu(k) - ry(k, r)) \md F_\rho(r)$ be the aggregate payoff function for an ambiguity-neutral central planner. Then $G_{RN}(k^*) > 0$.
\end{proposition}

\begin{proof}
Based on  the definition of the equilibrium threshold $k^*$, we have the equilibrium condition:
\begin{equation*}
\int_0^\infty \phi'(y(k^*, r)) [\mu(k^*) - ry(k^*, r)] \md F_\rho(r) = 0.
\end{equation*}
Define two functions of $r$ as follows
\begin{enumerate}
    \item $W(r) = \phi'(y(k^*, r))$. By Lemma \ref{lma:mono} , $y$ is strictly decreasing in $r$. As $\phi$ is strictly concave, $\phi'$ is a strictly decreasing function. Thus, $W(r)$ is a strictly increasing function of $r$.
    \item $V(r) = \mu(k^*) - ry(k^*, r)$. By Lemma \ref{lma:mono}, $ry$ is strictly increasing in $r$. Thus, $V(r)$ is a strictly decreasing function of $r$.
\end{enumerate}
Using Chebyshev's integral inequality, for a strictly increasing function $W(r)$ and a strictly decreasing function $V(r)$, we have
$
E[W(r)V(r)] < E[W(r)] E[V(r)].
$
Substituting the equilibrium condition $E[W(r)V(r)] = 0$,
\begin{equation*}
0 < \left( \int_0^\infty \phi'(y(k^*, r)) \md F_\rho(r) \right) \cdot \left( \int_0^\infty (\mu(k^*) - ry(k^*, r)) \md F_\rho(r) \right).
\end{equation*}
As $\phi$ is increasing, $\phi' > 0$, which implies the first integral is positive. Therefore, the second integral must also be positive, i.e.,
\begin{equation*}
G_{RN}(k^*) = \int_0^\infty (\mu(k^*) - ry(k^*, r)) \md F_\rho(r) > 0.
\end{equation*}
This completes the proof.
\end{proof}
\begin{corollary}
    For any aggregation function $\phi\in\mathcal{S}$, there is no barrier-type equilibrium.
\end{corollary}
\begin{proof}
    By Theorem \ref{thm:nes}, the barrier point $k^*$ must satisfy 
$$
 \int_0^\infty \phi'(d + y(k^*, r)) (\mathcal{L}y(k^*, r) - ry(k^*, r)) \md F_\rho(r) \leq 0,
$$
for any $d>0$. Dividing by $\phi^{\prime}(d)$ and letting $d\rightarrow\infty$, we have $\frac{\phi'(d + y(k^*, r))}{\phi'(d)}\rightarrow 1$ uniformly in $r\in \operatorname{Supp}(F_{\rho})$ as $\phi\in\mathcal{S}$ and $y$ is continuous in $r$ with $\operatorname{Supp}(F_{\rho})\subset[\underline{r},\hat{r}]$, which leads to  a contradiction as $d\rightarrow\infty$ by Proposition \ref{prop:neutral}.
\end{proof}
\subsection{Existence of Barrier-type Equilibrium in Linear and Exponential cases}
Having established that barrier-type equilibria do not exist for the class $\mathcal{S}$ under general surplus dynamics, we now turn our attention to aggregation functions where such equilibria exists. To derive explicit barrier points and carry out further analysis, we  focus on the diffusion model with constant coefficients for the remainder of this section.  That is, $\mu(x)\equiv \mu>0, \sigma(x)\equiv\sigma>0$ and under a dividend strategy $D=\{D_t\}_{t\ge0}$, the firm's surplus process evolves as
\begin{equation*}
dX_t^D=\mu \md t+\sigma\md W_t-\md D_t,\qquad X_{0-}^D=x>0.
\end{equation*}
By \eqref{eq:general_ode}, for any fixed $r \in \text{Supp}(F_\rho)$, the function $y(x, r)$ satisfies the following ODE in the no-transaction region:
\begin{equation}\label{eq:ODE_y}
\frac{1}{2} \sigma^2 y_{xx}(x, r) + \mu y_x(x, r) - r y(x, r) = 0, \quad x \in [0, k^*)
\end{equation}
with the initial value $y(0,r)=0$. Define the characteristic roots as follows
$$
\lambda_{\pm}(r):=\frac{-\mu\pm\sqrt{\mu^2+2\sigma^2r}}{\sigma^2}.
$$
Then the general solution to \eqref{eq:ODE_y} is given by $y(x, r) = A(r)(e^{\lambda_+ x} - e^{\lambda_- x})$. To satisfy the $C^1$-regularity at $x = k^*$, we require $y_x(k^*, r) = 1$ for any $r\in\operatorname{Supp}(F_{\rho})$, which yields
\begin{equation*}
A(r) = \frac{1}{\lambda_+ e^{\lambda_+ k^*} - \lambda_- e^{\lambda_- k^*}}.
\end{equation*}

To determine the equilibrium barrier $k^*$, we analyze the function defining the aggregate smooth-pasting condition \eqref{eq:aggregate_smooth_pasting}:
\begin{equation}\label{eq:smoothpast}
G(k) = \int_0^\infty \phi'(h(k, r)) \left( \mu - r h(k, r) \right) \md F_\rho(r),
\end{equation}
where $h(k, r)$ is defined by \footnote{The function $h(k, r)$ represents the expected discounted dividend at the candidate barrier $k$. } 
\begin{eqnarray}\label{h}
h(k, r):=\frac{e^{\lambda_+(r) k} - e^{\lambda_-(r) k}}{\lambda_+(r) e^{\lambda_+(r) k} - \lambda_-(r) e^{\lambda_-(r) k}},
\end{eqnarray} which represents the expected discounted dividend at the barrier $k$. As $\lambda_+ > 0 > \lambda_-$, we examine the asymptotic behavior of $G(k)$ at the boundaries of $(0, \infty)$:

As $\lim\limits_{k \to 0} h(k, r) =0$, given that $\mu > 0$ and $\phi$ is an increasing concave function with $\phi'(x) > 0$ for any $x>0$ , we have
$
\lim\limits_{k \to 0}G(k)  > 0.
$
Otherwise,  $\lim\limits_{k \to \infty} y(k, r) = \frac{1}{\lambda_+(r)}$ and thus
$
\lim_{k \to \infty} \left( \mu - r y(k, r) \right) = \mu - \frac{\mu + \sqrt{\mu^2 + 2\sigma^2 r}}{2} = \frac{\mu - \sqrt{\mu^2 + 2\sigma^2 r}}{2} < 0.
$
As this expression is strictly negative for all $r \in \text{Supp}(F_\rho)$ and $\phi' > 0$, it follows that $\lim_{k \to \infty} G(k) < 0$. Thus there exists a $k^*$ such that $G(k^*)=0$ as $G$ is continuous in $k$.

In the remainder of this subsection, we analyze two examples that automatically satisfy Conditions \eqref{additional assump1} and \eqref{additional assump2}: one is the case of linear ambiguity aggregation, i.e., $\phi(x)=x$, and the other is the case of exponential ambiguity aggregation, i.e., $\phi(x)=-\frac{1}{\alpha} e^{-\alpha x}$ with a constant aversion $\alpha>0$. In the first case, the derivative of $\phi$ is a constant, while in the second case, the derivative of $\phi$ takes an exponential form, and the initial dividend $d$ can be factored out. 
\subsubsection{Linear Aggregation and Comparison with Bounded Rate Case}\label{sec:linear}
Under linear ambiguity aggregation, the central planner is ambiguity-neutral with respect to the discount rate heterogeneity with $\phi'(x)\equiv 1$. The equilibrium condition \eqref{eq:smoothpast} simplifies to
\begin{equation}\label{eq:nosmoothambiguitybarrier}
\int_0^\infty \left( \mu - r h(k^*,r) \right) \md F_\rho(r)=\int_0^\infty \left( \mu - r \frac{e^{\lambda_+(r) k^*} - e^{\lambda_-(r) k^*}}{\lambda_+(r) e^{\lambda_+(r) k^*} - \lambda_-(r) e^{\lambda_-(r) k^*}} \right) \md F_\rho(r) = 0.
\end{equation}
\begin{theorem}\label{thm:linear equilibrium}
    There exists a unique barrier-type equilibrium in the diffusion model with linear ambiguity aggregation $\phi(x)=x$.
\end{theorem}
\begin{proof}
    We first calculate the derivative of $h(k, r) = \frac{e^{\lambda_+ k} - e^{\lambda_- k}}{\lambda_+ e^{\lambda_+ k} - \lambda_- e^{\lambda_- k}}$ with respect to $k$,
\begin{eqnarray}
\frac{\partial h}{\partial k} &=& \frac{(\lambda_+ e^{\lambda_+ k} - \lambda_- e^{\lambda_- k})^2 - (e^{\lambda_+ k} - e^{\lambda_- k})(\lambda_+^2 e^{\lambda_+ k} - \lambda_-^2 e^{\lambda_- k})}{(\lambda_+ e^{\lambda_+ k} - \lambda_- e^{\lambda_- k})^2}\nonumber \\
&=& \frac{\left( \lambda_+^2 e^{2\lambda_+ k} - 2\lambda_+ \lambda_- e^{(\lambda_+ + \lambda_-)k} + \lambda_-^2 e^{2\lambda_- k} \right) - \left( \lambda_+^2 e^{2\lambda_+ k} - \lambda_-^2 e^{(\lambda_+ + \lambda_-)k} - \lambda_+^2 e^{(\lambda_+ + \lambda_-)k} + \lambda_-^2 e^{2\lambda_- k} \right)}{(\lambda_+ e^{\lambda_+ k} - \lambda_- e^{\lambda_- k})^2}\nonumber \\
&=& \frac{(\lambda_+^2 - 2\lambda_+ \lambda_- + \lambda_-^2) e^{(\lambda_+ + \lambda_-)k}}{(\lambda_+ e^{\lambda_+ k} - \lambda_- e^{\lambda_- k})^2} = \frac{(\lambda_+ - \lambda_-)^2 e^{(\lambda_+ + \lambda_-)k}}{(\lambda_+ e^{\lambda_+ k} - \lambda_- e^{\lambda_- k})^2} > 0.\label{h1}
\end{eqnarray}
As $\frac{\partial h}{\partial k} > 0$, the derivative of $G(k)$ is $-\int_0^{\infty}r\cdot\frac{\partial h(k,r)}{\partial k}\md F_{\rho}(r)<0$. Thus, $G(k)$ is continuous and strictly decreases in $(0, \infty)$. Given $G(0) = \mu > 0$ and $G(\infty) < 0$, there exists a unique root $k^* \in (0, \infty)$ such that $G(k^*) = 0$.

Next we prove that the barrier-type dividend law with the barrier point $k^*$ is an equilibrium dividend law.
The regularity conditions ($y \in C^1(\mathbb{R}_+) \cap C^2(NT)$) and the ODE system \eqref{eq:y_system} are satisfied by construction of the function $y(x, r)$. 
The monotonicity of $y$ with respect to $x$ and $r$ and the growth condition \eqref{condition: growth} of $y$ are easily verified. To prove that the barrier-type dividend law with the barrier point $k^*$ is an equilibrium, we only need to verify the variational inequality \eqref{variational inequality}. Specifically, within the no-transaction region $\mathcal{NT} = (0, k^*)$, the infinitesimal generator satisfies
\begin{equation*}
    \mathcal{L}y(x, r) - r y(x, r) = 0.
\end{equation*}
In the pay region $\mathcal{P} = [k^*, \infty)$, we have $y(x, r) = x - k^* + y(k^*, r)$. The generator term becomes
\begin{equation*}
    \mathcal{L}y - ry = \mu - r(x - k^* + y(k^*, r)).
\end{equation*}
As $r > 0$, this expression is decreasing in $x$. Given the aggregate equilibrium condition at $k^*$, we have
\begin{equation*}
    \int_0^\infty (\mu - r y(k^*, r)) \md F_\rho(r) = 0,
\end{equation*}
it follows that for all $x \in \mathcal{P}$, the aggregate generator condition $\int_0^{\infty}  (\mathcal{L}y - ry) \md F_\rho(r) \leq 0$ is satisfied.  To conclude the proof, we must verify the variational inequality for the marginal payoff in $\mathcal{NT}$,
\begin{equation*}
    H(x) := \int_0^\infty (1 - y_x(x,r)) \md F_\rho(r) \leq 0, \quad \forall x \in (0, k^*).
\end{equation*}
As $H(k^*)=0,\;H^{\prime}(k^*)=-\int_0^\infty y_{xx}(k^*-,r)=0$ and $y_{xxx}=\frac{\lambda_+^3e^{\lambda_+ x} - \lambda_-^3e^{\lambda_- x}}{\lambda_+ e^{\lambda_+ k} - \lambda_- e^{\lambda_- k}}>0$, we have $H^{\prime}(x)>0$ on $(0,k^*)$ and thus $H(x)< 0$ on $(0,k^*)$. Therefore, the barrier-type dividend law with the barrier point $k^*$ is exactly an equilibrium dividend law.
\end{proof}
We now compare our results with those of \cite{zhao2014dividend}, which investigate equilibrium dividend strategies under non-exponential discounting. A fundamental difference lies in the admissible strategies:  we consider singular dividend control $D_t$, which allows for lump-sum payments (jumps in the surplus process), while \cite{zhao2014dividend} restrict dividends to be paid at a bounded rate $\ell_t \in [0, M]$, meaning the cumulative dividend process $D_t = \int_0^t \ell_s ds$ is absolutely continuous.

Consider the mixture of exponential discount example in \cite{zhao2014dividend}:
\begin{equation}\label{eq:mixture_h}
h(t) = \sum_{i=1}^N \omega_i e^{-\delta_i t}, \quad t \ge 0.
\end{equation}
This corresponds to our ``linear ambiguity aggregation" case where the ambiguity distribution $F_\rho$ is a discrete measure with weights $\omega_i$ at discount rates $r = \delta_i>0$. 

In \cite{zhao2014dividend}, the equilibrium strategy $\ell^*(x)$ is also shown to be a barrier strategy. Specifically, the equilibrium law depends on the magnitude of the dividend rate bound $M$. When the dividend payment capacity $M$ is sufficiently large, there exists a unique positive barrier $b > 0$ such that

\begin{equation*}
\ell^*(x) = 
\begin{cases} 
0, & 0 \le x < b, \\
M, & x \ge b.
\end{cases}
\end{equation*}
The barrier $b$ is determined by 
\begin{equation}\label{eq:Zhao_barrier_cond}
\sum_{i=1}^N \omega_i V_i'(b,M) = 1,
\end{equation}
where $V_i(x)$ is the component of the value function associated with each discount rate $\delta_i$\footnote{The function $V_i$ plays the same role as $y_i$ in our problem.}. Based on \cite{zhao2014dividend} , the derivative $ V_i'$ at the barrier can be expressed as
\begin{equation}\label{eq:Vi_prime_Zhao}
V_i'(b,M) = \frac{M \theta_{i3}}{\delta_i} \cdot \frac{\lambda_+(\delta_i) e^{\lambda_+(\delta_i)b} -\lambda_-(\delta_i) e^{\lambda_-(\delta_i)b}}{(\lambda_+(\delta_i) + \theta_{i3}) e^{\lambda_+(\delta_i)b} - (\lambda_-(\delta_i) + \theta_{i3}) e^{\lambda_-(\delta_i)b}},
\end{equation}
where $-\theta_{i3}$ is the negative root of $\frac{1}{2}\sigma^2 y^2 + (\mu-M) y - \delta_i = 0$.

We now analyze the behavior of the barrier $b$ as the dividend rate limit $M \to \infty$ and present a rigorous derivation in Theorem \ref{thm: nosmoothconvergence} to show that the solutions to the equations \eqref{eq:Zhao_barrier_cond} in \cite{zhao2014dividend} converge to the solutions to \eqref{eq:nosmoothambiguitybarrier} in our work.
This analysis confirms that the equilibrium singular dividend strategy, which is induced by a barrier-type equilibrium dividend law at $k^*$, is the limit of the bounded-rate equilibrium regular strategy as the payment capacity $M\rightarrow\infty$.
\begin{theorem}\label{thm: nosmoothconvergence}
Let $b(M)$ be the unique equilibrium barrier point for the bounded dividend rate case with a sufficiently large bound $M$, satisfying the aggregate condition $\sum\limits_{i=1}^N \omega_i V_i'(b, M) = 1$. Let $k^*$ be the equilibrium threshold for our singular dividend control problem under the linear aggregation case, satisfying the equation $G(k^*) = \sum\limits_{i=1}^N \omega_i (\mu - \delta_i h(k^*, \delta_i)) = 0$. Then we have 
$ \lim\limits_{M \to \infty} b(M) = k^* . $
\end{theorem}

\begin{proof}
In the bounded rate case, from \eqref{eq:Vi_prime_Zhao}, the derivative of the value function at the barrier $b$ is given by
$$ V_i'(b, M) = \frac{M\theta_{i3}}{\delta_i} \cdot \frac{1}{1 + \theta_{i3} h(b, \delta_i)},  $$
where $h$
is defined by \eqref{h}. As $-\theta_{i3}$ is the negative root of $\frac{1}{2}\sigma^2 \theta^2 + (\mu-M)\theta - \delta_i = 0$, it satisfies the following asymptotic behavior as $M \to \infty$,
$$ \theta_{i3} = \frac{2\delta_i}{\sqrt{(M-\mu)^2+2\sigma^2\delta_i}+M-\mu} \implies \frac{M\theta_{i3}}{\delta_i} = 1 + \frac{\mu}{M} + O(M^{-2}) . $$
Performing a Taylor expansion of $V_i'(b, M)$ in terms of $1/M$ yields
$$ V_i'(b, M) = \left( 1 + \frac{\mu}{M} + O(M^{-2}) \right) \left( 1 - \frac{\delta_i h(b, \delta_i)}{M} + O(M^{-2}) \right) = 1 + \frac{\mu - \delta_i h(b, \delta_i)}{M} + O(M^{-2}) .$$
Substituting this into the equilibrium condition $\sum\limits_{i=1}^N \omega_i V_i'(b(M), M) = 1$ and noting that $\sum\limits_{i=1}^N \omega_i = 1$, we obtain
$$ 1 + \frac{1}{M} \sum_{i=1}^N \omega_i (\mu - \delta_i h(b(M), \delta_i)) + O(M^{-2}) = 1 . $$
Multiplying by $M$ and recalling that $G(k) = \sum_{i=1}^N \omega_i (\mu - \delta_i h(k, \delta_i))$, it follows that,  for sufficiently large $M\rightarrow\infty$,
$ G(b(M)) \rightarrow 0. $
As $G(k^*)=0$ and $G$ is strictly decreasing, the inverse function $G^{-1}$ exists and is continuous at the point $0$.  Then
$$ \lim_{M \to \infty} b(M) = \lim_{M \to \infty} G^{-1}\big(G(b(M))\big) = G^{-1}\big(\lim_{M \to \infty} G(b(M))\big) = G^{-1}(0) = k^* .$$
Therefore, the bounded rate equilibrium barrier $b(M)$  converges to the singular barrier  $k^*$ .
\end{proof}

\subsubsection{Exponential Aggregation Case}\label{sec:expon}

In this subsection, we consider the case where the central planner's preferences are characterized by an exponential  ambiguity aggregation function
\begin{equation*}
    \phi(x) = -\frac{1}{\alpha} e^{-\alpha x}, \quad \alpha > 0,
\end{equation*}
where $\alpha\ge 0$ represents the central planner's absolute ambiguity aversion. According to \eqref{eq:smoothpast}, the equilibrium barrier point $k^*$ must satisfy
\begin{equation*}
G(k^*)=\int_0^\infty e^{-\alpha h(k^*, r)} \left( \mu - r h(k^*, r) \right) \md F_\rho(r) = 0.
\end{equation*}

\begin{assumption}\label{assum:alpha_constraint}
The coefficient of absolute ambiguity aversion $\alpha$ satisfies
$
    \alpha < \frac{\mu + \sqrt{\mu^2 + 2\sigma^2 \underline{r}}}{\sigma^2}.
$
\end{assumption}

\begin{lemma}\label{lem:G_monotonicity}
Under the exponential ambiguity aggregation function and Assumption \ref{assum:alpha_constraint}, define the aggregate equilibrium function $G(k)$ by
\begin{equation*}
    G(k) = \int_{0}^{\infty} e^{-\alpha h(k, r)} \left( \mu - r h(k, r) \right) \md F_{\rho}(r).
\end{equation*}
Then
\begin{enumerate}
    \item $G(k)$ is continuous and strictly decreasing on $(0, \infty)$;
    \item There exists a unique root $k^* \in (0, \infty)$ such that $G(k^*) = 0$.
\end{enumerate}
\end{lemma}

\begin{proof}
Differentiating $G(k)$ with respect to $k$ yields
\begin{equation*}
    G'(k) = \int_{0}^{\infty} e^{-\alpha h(k, r)} \frac{\partial h(k, r)}{\partial k} \left[ -\alpha(\mu - r h(k, r)) - r \right] dF_{\rho}(r).
\end{equation*}
Given that $\frac{\partial h(k, r)}{\partial k} > 0$ (see \eqref{h}) and $e^{-\alpha h(k, r)} > 0$ for all $k, r > 0$, the condition $G'(k) < 0$ is satisfied if
\begin{equation*}
    \alpha(\mu - r h(k, r)) + r > 0, \quad \forall k > 0, \,r \in \mathrm{Supp}(F_{\rho}).
\end{equation*}
As $h(k, r)$ is strictly increasing in $k$ and $\lim\limits_{k \to \infty} h(k; r) = 1/\lambda_{+}(r)$, we compute
\begin{equation*}
   r h(k, r)-\mu\le \frac{r}{\lambda_{+}(r)} - \mu = \frac{\sqrt{\mu^2 + 2\sigma^2 r} - \mu}{2}.
\end{equation*}
Thus the sufficiency condition for $\alpha$ is given by
\begin{equation*}
    \alpha \leq \frac{r}{\frac{r}{\lambda_{+}(r)} - \mu} = \frac{\mu + \sqrt{\mu^2 + 2\sigma^2 r}}{\sigma^2}.
\end{equation*}
As the right-hand side of the last inequality is an increasing function of $r$, Assumption \ref{assum:alpha_constraint} ensures that $G'(k) < 0$ holds globally for all $k \in (0, \infty)$.
Hence, $G=0$ has a unique root  $k^* \in (0, \infty)$.
\end{proof}

\begin{lemma}\label{lem:G_tilde_x_monotonicity}
Let $k^* > 0$ be the unique equilibrium barrier. For any $x \in (0, k^*]$, define the state-dependent aggregate function $\tilde{G}(x)$ by
\begin{equation*}
    \tilde{G}(x) = \int_{0}^{\infty} e^{-\alpha y(x, r)} \left( \mu - r y(x, r) \right) \md F_{\rho}(r),
\end{equation*}
where $y(x, r)$ is the expected discounted dividend  given the barrier-type dividend law with fixed barrier $k^*$. Under Assumption \ref{assum:alpha_constraint}, the function $\tilde{G}(x)$ is strictly decreasing in $ (0, k^*]$.
\end{lemma}

\begin{proof}
Differentiating $\tilde{G}(x)$ with respect to the state $x$ yields
\begin{align*}
    \frac{d\tilde{G}(x)}{dx} &= \int_{0}^{\infty} \frac{d}{dx} \left[ e^{-\alpha y(x, r)} \left( \mu - r y(x, r) \right) \right] \md F_{\rho}(r) = \int_{0}^{\infty} \left[ -\alpha e^{-\alpha y} y_x (\mu - r y) + e^{-\alpha y} (-r y_x) \right] \md F_{\rho}(r) \\
    &= \int_{0}^{\infty} e^{-\alpha y} y_x \left[ \alpha(r y(x, r) - \mu) - r \right] \md F_{\rho}(r).
\end{align*}
As $e^{-\alpha y} > 0$ and $y_x > 0$, the derivative is non-positive if the term in the brackets satisfies
\begin{equation*}
    \alpha(r y(x, r) - \mu) \leq r, \quad \forall r \in \mathrm{Supp}(F_{\rho}).
\end{equation*}
Because $y(x, r)$ is strictly increasing in $x$,  $\alpha(ry(x, r) - \mu)$ attains its maximum  $x = k^*$, where it equals $\alpha(rh(k^*,r) - \mu)$.  Consequently, by the same argument as  in Lemma \ref{lem:G_monotonicity}, $\frac{d\tilde{G}(x)}{dx} < 0$ in $(0,k^*)$.
\end{proof}
Now we can prove that the barrier-type dividend law with barrier point $k^*$ is actually an equilibrium dividend law.
\begin{theorem}
Under Assumption \ref{assum:alpha_constraint}, the barrier strategy $\hat{\mathcal{D}}$ with the barrier point $k^*$ is an equilibrium singular dividend law in the exponential aggregation case.
\end{theorem}

\begin{proof}
Based on the same argument as in Theorem \ref{thm:linear equilibrium}, it remains to verify the variational inequality for the marginal payoff in $\mathcal{NT}$,
\begin{equation*}
    H(x) := \int_0^\infty \phi'(y(x; r)) (1 - y_x(x; r)) dF_\rho(r) \leq 0, \quad \forall x \in [0, k^*).
\end{equation*}
Differentiating $H(x)$ and substituting $y_{xx} = \frac{2}{\sigma^2}(ry - \mu y_x)$, we obtain
\begin{align*}
    H'(x) &= \int_0^\infty \left[ \phi''(y) y_x (1 - y_x) + \phi'(y) (-y_{xx}) \right] \md F_\rho(r) = \int_0^\infty e^{-\alpha y} \left( \alpha y_x^2 - \alpha y_x + \frac{2\mu}{\sigma^2} y_x - \frac{2r}{\sigma^2} y \right) \md F_\rho(r).
\end{align*}
Suppose $H(x) = 0$ for some $x \in [0, k^*)$. Then
\begin{equation}\label{eq:H_zero_cond}
    \int_0^\infty e^{-\alpha y} \md F_\rho(r) = \int_0^\infty e^{-\alpha y} y_x \md F_\rho(r).
\end{equation}
By the Cauchy-Schwarz inequality,
\begin{equation*}
    \left( \int_0^\infty e^{-\alpha y} y_x \md F_\rho(r) \right)^2 \leq \left( \int_0^\infty e^{-\alpha y} y_x^2 \md F_\rho(r) \right) \left( \int_0^\infty e^{-\alpha y}\md F_\rho(r) \right).
\end{equation*}
Substituting \eqref{eq:H_zero_cond} into the above yields
\begin{equation*}
    \int_0^\infty e^{-\alpha y} y_x^2 \md F_\rho(r) \geq \int_0^\infty e^{-\alpha y} \md F_\rho(r) = \int_0^\infty e^{-\alpha y} y_x \md F_\rho(r).
\end{equation*}
Consequently, the terms involving $\alpha$ in $H'(x)$ satisfy $\int e^{-\alpha y} (\alpha y_x^2 - \alpha y_x) \md F_\rho (r)\geq 0$. Then
\begin{equation*}
    H'(x) \geq \frac{2}{\sigma^2} \int_0^\infty e^{-\alpha y} (\mu y_x - ry) \md F_\rho = \frac{2}{\sigma^2} \int_0^\infty e^{-\alpha y} (\mu - ry) \md F_\rho > 0,
\end{equation*}
where the final inequality follows from the fact that $\tilde{G}(x) = \int_0^{\infty} e^{-\alpha y}(\mu - ry)\md F_{\rho}(r)$ is strictly decreasing by Lemma \ref{lem:G_tilde_x_monotonicity} and $\tilde{G}(k^*) = 0$.
At the boundary $k^*$, we have $H(k^*) = 0$ and $H'(k^*) = 0$. Furthermore, 
$$H''(k^*)= \int_0^\infty e^{-\alpha y} ( - y_{xxx}) dF_\rho< 0.$$ 
Hence $H(x) < 0$ in a left-neighborhood of $k^*$. If $H(x)$ were to exceed 0 elsewhere in $[0, k^*)$, it would have to cross 0 from below at some $x_0$, implying $H'(x_0) \leq 0$. This contradicts the implication proved above that $H(x_0)=0$ entails $H'(x_0) > 0$. Thus, $H(x) \leq 0$ for all $x \in [0, k^*)$.
\end{proof}

\section{Numerical Examples}\label{sec:Numerical examples}
In this section, we present some numerical examples to illustrate the theoretical results and examine the impact of discount-rate ambiguity on the equilibrium barrier. We consider the diffusion model with the following baseline parameters. The drift of the surplus process is $\mu=0.1$, and the volatility is $\sigma=0.2$. The discount rate takes only two possible values: $r_1=0.03$ (patient members) and $r_2=0.07$ (impatient members), and follows the  two‑point distribution $\mathbb{P}(\rho = r_1)=\omega$ and $\mathbb{P}(\rho = r_2)=1-\omega$. We consider three values of the weight $\omega=0.2,0.5,0.8$ and also the extreme cases $\omega=0$ (all members impatient) and $\omega=1$ (all members patient) as benchmarks. We consider two specifications of the ambiguity aggregation function: the ambiguity-neutral case $\phi(x)=x$ (linear aggregation) and the exponential aggregation case $\phi(x)=-\frac{1}{\alpha}e^{-\alpha x}$. For each case, we compute the equilibrium barrier $k^*$ from the corresponding equilibrium condition derived above,  for the linear aggregation case,
$$
\omega (\mu-r_1h(k^*,r_1))+(1-\omega)(\mu-r_2h(k^*,r_2))=0,
$$
which is equivalent to 
$$
\omega r_1 h(k^*,r_1)+(1-\omega)r_2h(k^*,r_2)=\mu,
$$
and for the exponential ambiguity aggregation case,
$$
\omega e^{-\alpha h(k^*,r_1)}(\mu-r_1h(k^*,r_1))+(1-\omega)e^{-\alpha h(k^*,r_2)}(\mu-r_2 h(k^*,r_2))=0,
$$
where $h(k,r):=\frac{e^{\lambda_+(r) k} - e^{\lambda_-(r) k}}{\lambda_+(r) e^{\lambda_+(r) k} - \lambda_-(r) e^{\lambda_-(r) k}}$ and $\lambda_{\pm}(r):=\frac{-\mu\pm\sqrt{\mu^2+2\sigma^2r}}{\sigma^2}.$

 \begin{figure}[htbp]
  \centering
  \begin{minipage}{0.5\textwidth}
   \centering
  \includegraphics[totalheight=5.9cm]{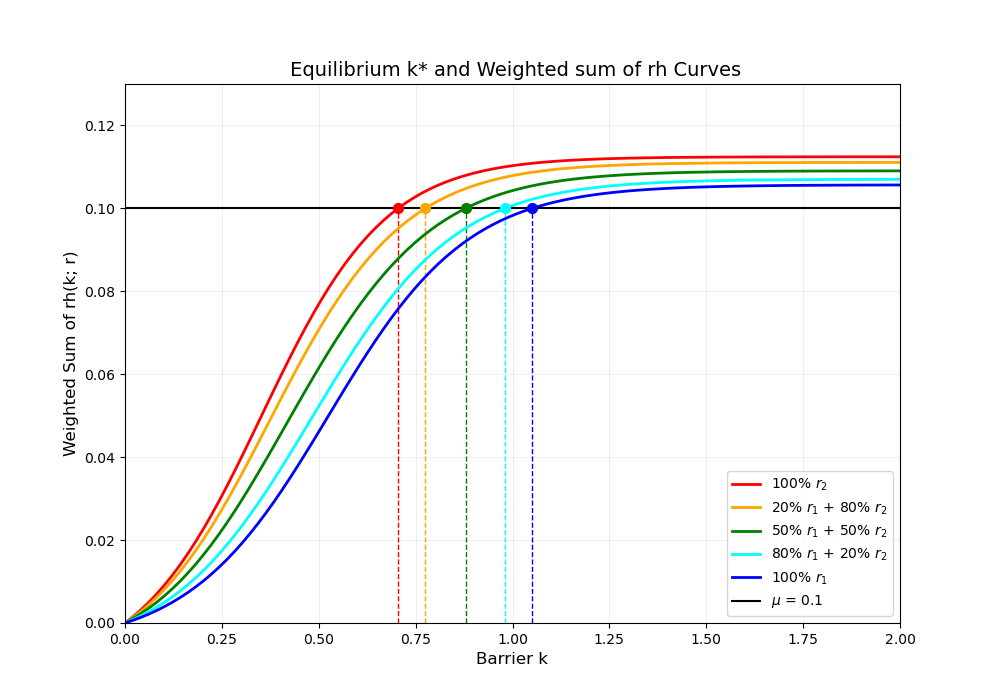}
\caption{Weighted sum of $rh(k,r)$ and equilibrium barrier $k^*$ under different distributions.}
  \label{fig:barrier}
  \end{minipage}\hfill
 \begin{minipage}{0.5\textwidth} 
  \centering
  \includegraphics[totalheight=5.9cm]{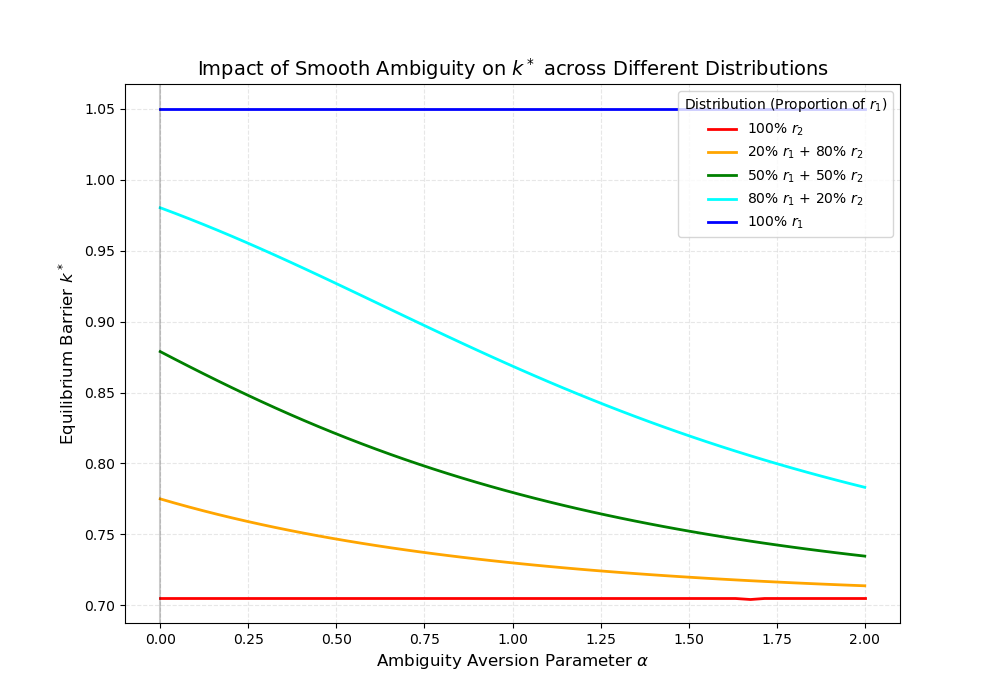}
\caption{Equilibrium barrier $k^*$ under different discount rate distributions and ambiguity aversion $\alpha$.}
  \label{fig:expbarrier}
  \end{minipage}\hfill
\end{figure}

Figure \ref{fig:barrier} illustrates the equilibrium barrier $k^*$ for the linear aggregation case ($\phi(x)=x$) as the proportion of patient members (i.e., the proportion of members with discount rate $r_1$) changes. The barrier $k^*$ is determined by the intersection of the horizontal line $\mu=0.1$ with the curve of the weighted sum of $rh(k,r)$. The resulting intersection points show that $k^*$ increases monotonically with $\omega$. When the group consists entirely of impatient members $(\omega=0)$, the barrier is lowest, reflecting a strong preference for immediate dividend payouts. As the proportion of patient members grows, the central planner becomes more willing to retain surplus and delay dividends, thus raising the barrier. 

Figure \ref{fig:expbarrier} plots the equilibrium barrier $k^*$ under exponential ambiguity aggregation with the central planner's ambiguity aversion parameter $\alpha\in [0,2]$\footnote{It is easy to verify that $\alpha\in[0,2]$ satisfies Assumption \ref{assum:alpha_constraint}.} for different discount distributions. In the homogeneous cases $\omega=0,1$, the barrier is constant with respect to $\alpha$. This is expected because when the discount rate is deterministic, the outer integral in the equilibrium condition collapses, and ambiguity aversion has no effect. For the mixed distributions ($\omega=0.2, 0.5,0.8$), the barrier decreases as $\alpha$ increases. This indicates that a more ambiguity‑averse central planner places greater weight on the high discount rate (impatient members), effectively behaving as if the group were more impatient. Consequently, the equilibrium barrier is lowered, and dividends are paid out earlier.

As the central planner aggregates heterogeneous discount rates, a natural question arises: what single discount rate, if assigned to a representative agent who optimizes in the classical setting, would yield an optimal barrier that coincides with the equilibrium barrier $k^*$ obtained under aggregation? In other words, we seek a virtual discount rate $r_{virtual}$ such that the optimal barrier for a single agent with discount rate $r_{virtual}$ equals $k^*$. According to our derived barrier condition,  the corresponding virtual discount rate $r_{virtual}$ can be implicitly expressed by the relation
\[
k = \frac{\sigma^2}{ \sqrt{\mu^2 + 2\sigma^2 r}} \ln \left( \frac{ \sqrt{\mu^2 + 2\sigma^2 r}+\mu}{  \sqrt{\mu^2 + 2\sigma^2 r}-\mu} \right).
\]
This is the same as the optimal barrier condition for a time-consistent dividend problem without discount rate ambiguity, which was first derived in \cite{taksar2000optimal}.
Conversely, given the barrier $k^*$ obtained from our heterogeneous model, we solve the above relation to obtain a unique virtual discount rate $r_{virtual}$. This virtual rate provides an intuitive summary of the group's collective impatience, allowing us to compare different heterogeneous compositions and ambiguity attitudes on a common scale.
\begin{figure}[htbp]
  \centering
  \includegraphics[width=0.75\textwidth]{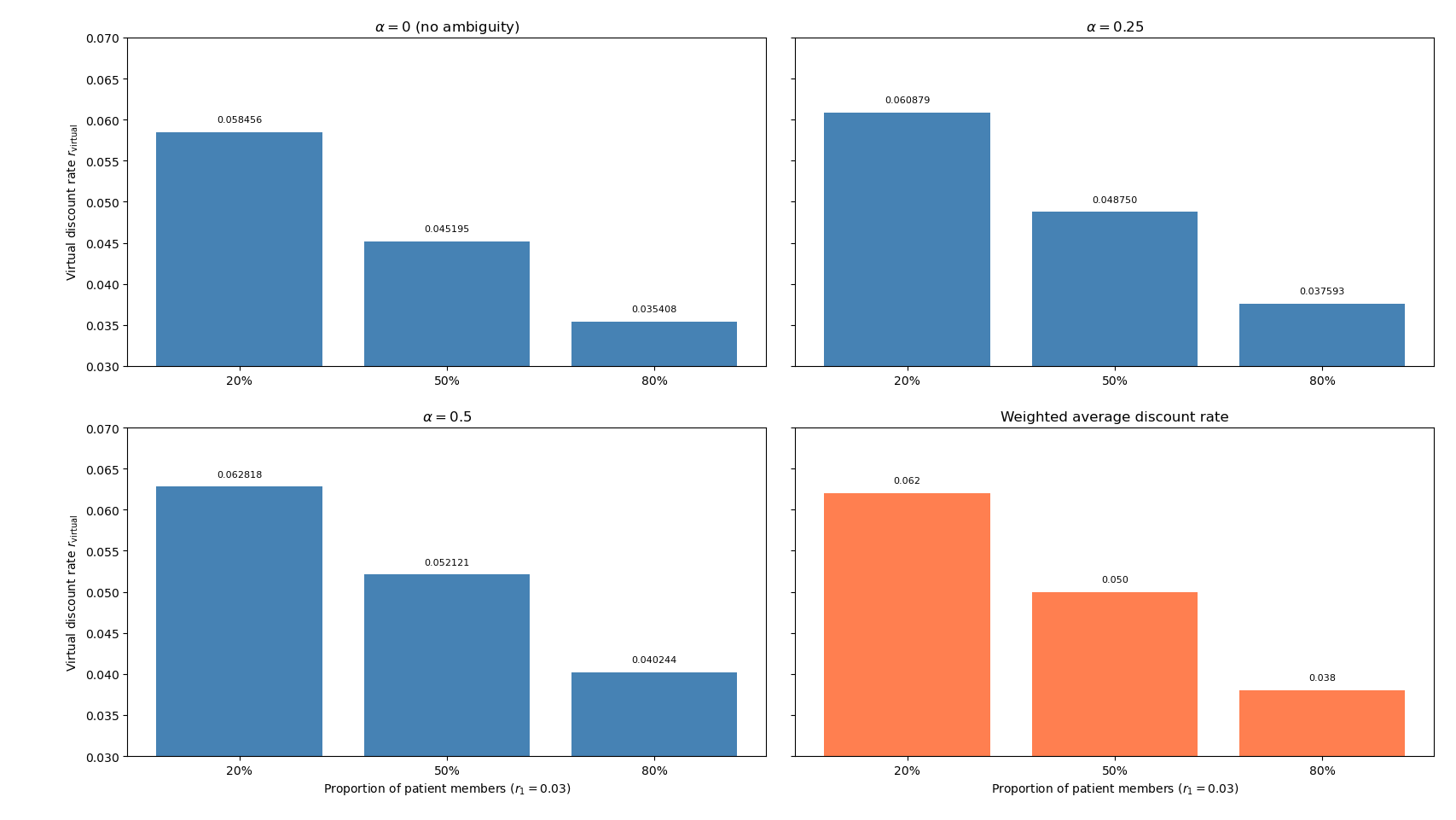}
\caption{ Virtual discount rates under different $\alpha$ and the weighted average discount rate.}
  \label{fig:virtual_vs_weighted}
\end{figure}

Figure \ref{fig:virtual_vs_weighted} compares the virtual discount rate $r_{virtual}$ under different levels of ambiguity aversion with the benchmark weighted-average discount rate $\tilde{r} = \omega r_1 + (1-\omega) r_2$. A clear pattern emerges: in the absence of ambiguity aversion ($\alpha = 0$), the virtual discount rate lies consistently below the weighted average. As $\alpha$ increases, the virtual rate rises, and for $\alpha = 0.5$ it exceeds the weighted average. This behavior can be understood by recognizing two opposing forces.

On the one hand, in the ambiguity-neutral case ($\alpha = 0$), the function $rh(k,r)$ turns out to be concave. By Jensen's inequality, $\mathbb{E}[r h(k,r)] \le \mathbb{E}[r]\, h(k,\mathbb{E}[r])$. The equilibrium condition $\mathbb{E}[r h(k,r)] = \mu$ then implies $\mu \le \tilde{r}\,h(k,\tilde{r})$. As the virtual discount rate $r_{virtual}$ satisfies $\mu = r_{virtual}\,h(k,r_{virtual})$ and $q(r)=rh(k,r)$ is increasing, it follows that $r_{virtual} \le \tilde{r}$. Hence, the virtual rate falls below the weighted average — the concavity effect makes the central planner appear more patient (i.e., inclined to delay dividends).

On the other hand, when ambiguity aversion is introduced, the exponential weight $e^{-\alpha y(k,r)}$ amplifies the contribution of small $y$ values, which correspond to high discount rates (impatient members). This increases the effective weight of impatient members in the equilibrium condition, pushing the virtual discount rate upward and making the collective decision more impatient.
Thus, the equilibrium strategy is shaped by two competing forces: diversity of discount rates makes it more patient, while ambiguity aversion makes it more impatient. When ambiguity aversion is weak ($\alpha = 0.25$), the diversity effect dominates, and the virtual rate remains below the weighted average. When ambiguity aversion is sufficiently strong ($\alpha = 0.5$), ambiguity aversion dominates the diversity effect, and the virtual rate exceeds the weighted average.

\vspace{0.2in}\noindent
\textbf{Acknowledgements}: The authors acknowledge the support from the National Natural Science Foundation of China (Grant Nos. 12271290 and  12371477).

\bibliographystyle{plainnat}
\bibliography{reference}
\appendix

\end{document}